%% file: master.tex
\documentclass[11pt,a4paper]{amsart}
\usepackage{amsfonts,amsmath,amssymb,amsthm}
\usepackage{times}
\usepackage{graphics,graphicx}
\usepackage{color}
\usepackage{enumerate}
\usepackage[shortlabels]{enumitem}

\usepackage[utf8]{inputenc}
\usepackage{bbm, dsfont}
\usepackage{mathtools}
\mathtoolsset{showonlyrefs=true}

\renewcommand{\epsilon}{\varepsilon}

\newcommand{\bdirsum}[2]{\underset{#2}{\overrightarrow{\bigoplus}}#1}
\newcommand{\dirsum}[2]{\underset{#2}{\overrightarrow{\oplus}}#1}

\usepackage{bm}

\DeclareMathOperator{\id}{id}

\newcommand{\bW}{\mathbf{W}}

\newcommand{\E}{\mathbb{E}}
\newcommand{\N}{\mathbb{N}}
\newcommand{\pP}{\mathbb{P}}

\newcommand{\R}{\mathbb{R}}

\newcommand{\cD}{{\mathcal D}}

\newcommand{\cI}{{\mathcal I}}

\newcommand{\cP}{{\mathcal P}}
\newcommand{\cQ}{{\mathcal Q}}

\newcommand{\cS}{{\mathcal S}}
\newcommand{\cT}{{\mathcal T}}
\newcommand{\cZ}{{\mathcal Z}}

\newcommand{\sC}{{\mathsf C}}

\newcommand{\sT}{{\mathsf T}}
\newcommand{\sP}{{\mathsf P}}
\newcommand{\sE}{{\mathsf E}}

\DeclareMathOperator{\Leb}{Leb}

\newtheorem{theorem}{Theorem}[section]{\bf}{\it}
{\bf}{\it}
{\bf}{\it}
{\bf}{\it}
\newtheorem{corollary}[theorem]{Corollary}{\bf}{\it}
\newtheorem{example}[theorem]{Example}{\it}{\rm}
\newtheorem{lemma}[theorem]{Lemma}{\bf}{\it}
\newtheorem{remark}[theorem]{Remark}{\it}{\rm}
\newtheorem{definition}[theorem]{Definition}{\bf}{\it}

\title[Decomposition of tournament limits]{Decomposition of tournament limits}
\author[Erik ~Th\"ornblad]{Erik Th\"ornblad}
 \address{Department of Mathematics, Uppsala University, Box 480, S-75106 Uppsala, Sweden.}
 \email{erik.thornblad@math.uu.se}
 \date{\today}

\begin{document}

\begin{abstract}
The theory of tournament limits and tournament kernels (often called graphons) is developed by extending common notions for finite tournaments to this setting; in particular we study transitivity and irreducibility of limits and kernels. We prove that each tournament kernel and each tournament limit can be decomposed into a direct sum of irreducible components, with transitive components interlaced. We also show that this decomposition is essentially unique.
\end{abstract}

\maketitle

\input{master_introduction}

\input{master_preliminaries}
\input{master_tourdecomp}
\input{master_kernels}

\input{master_transitive}

%\input{master_directsums} not necessary anymore
%\input{master_kerneldecomp}
\input{master_kerneldecomp4}
\input{master_densityformula3}
%\input{master_kerneldecomp3}
%\input{master_densityformula}
%\input{master_densityformula2}
\input{master_irreducible}
\input{master_degdist}
\input{master_acknowledgements}

\bibliographystyle{abbrv}

\bibliography{../../references}
\end{document}

%% file: master_introduction.tex
%%%%%%%%%%%%%%%%%%%%%%%%%%%%%%%%%%%%%%
\section{Introduction}
%%%%%%%%%%%%%%%%%%%%%%%%%%%%%%%%%%%%%%

% One standard decomposition result regarding tournaments is that the vertex set of any tournament can be uniquely partitioned into subsets, each of which induces either an irreducible tournament (there are directed paths connecting any pair of vertices in both directions) or an acyclic tournament (there are no cyclic subgraphs), such that these components can be linearly ordered with the edges between the components following the linear ordering. 

Informally speaking, graph limits are abstract limit objects of graph sequences $(G_n)_{n=1}^{\infty}$ with $v(G_n)\to \infty$ associated to the convergence of subgraph densities of the sequence $(G_n)_{n=1}^{\infty}$. This theory was initiated for undirected graphs for undirected graphs in \cite{LovaszSzegedy2006} and later developed in among others \cite{Borgs2008, Borgs2012}. The limit objects can be non--trivial only if $|E(G_n)|=\Theta(v(G_n)^2)$, so in this sense the theory of graph limits concerns itself with sequences of dense graphs, although attempts have been made to extend this notion to the sparse setting, see e.g. \cite{Lovaszbook}, which also provides a general overview of the theory of graph limits. For dense graph limits, the limit of a sequence of graphs is not sensitive to perturbations. In fact, if $(G_n)_{n=1}^{\infty}$ is a sequence of graphs that converging to some graph limit, it can be shown that this converges to the same limit even if one, for each $n\geq 1$, adds or removes up to $o(v(G_n)^2)$ edges to the graph $G_n$. It is common to try to extend results known about finite graphs to the setting of graph limits, and the non--sensitivity to perturbations of the sequence, coupled with additional analytical tools that become available, often makes results easier to prove in the limit case. However, it should be mentioned that it is not always the case that one can carry out sensible extension of the theory of finite graphs to graph limits.

In the undirected case, each graph limit can be represented by a so--called \emph{kernel} or \emph{graphon}, which is a symmetric function $[0,1]^2\to [0,1]$. The correspondence between kernels and the graph limits they represent is highly--nontrivial, and in general there are many kernels representing the same graph limit. The \emph{cut norm} defined on the set of such symmetric functions provides a framework of determining when two kernels represent the same graph limit. We shall avoid using the cut norm in this paper, so we do not pursue this matter further here. Of interest to us is that the theory of graph limits was extended to directed graphs by Diaconis and Janson \cite{DiaconisJanson}, who showed that digraph limits appear as extreme points in the set of distributions of exchangeable arrays (this holds also for undirected graphs), demonstrating a connection between graph limit theory and probability theory. Similar to the undirected case, they showed that each directed graph limit can be represented by a so--called kernel, which in the directed case consists of a quintuple of functions. 

A tournament is a complete directed graph, and we shall be restricting our attention to the setting of tournament limits. In this case, the corresponding (tournament) kernels have a particularly easy form; indeed, we shall see later that it suffices to consider functions $W:[0,1]^2\to [0,1]$ satisfying $W(x,y)+W(y,x)=1$. A trivial fact about tournaments is that these are dense directed graphs, so we can expect non--trivial limits and kernels of tournament sequences to appear.

Although we will develop the general framework of tournament limits and extend the notions of irreducibility and transitivity of finite tournaments to the setting of tournament kernels and tournament limits, our main aim is a certain decomposition result. Decomposition of finite graphs into `components' is a well--studied problem. The paper \cite{CourcelleDelhomme2008} contains several decomposition results for both undirected and directed graphs, including infinite such.  Generally speaking, a decomposition of a graph is a partition of its edge set or its vertex sets into subsets, typically under the requirement that each part of the partition satisfies some desired property (by themselves or pairwise).  The simplest result of this type is perhaps the decomposition of a finite graph into its connected components, which means a partition of the vertex set such that each part induces a connected subgraph, and there are no edges between different components. This type of decomposition was extended to (undirected) graph kernels and graph limits by Janson \cite{Janson2008}. 

To be precise, our aim is to extend the following decomposition result mentioned by Moon \cite{Moon1968} for finite tournaments. Given any tournament, its vertex set can be uniquely partitioned into subsets, each of which induces either an irreducible tournament (there are directed paths in both directions between any pair of vertices) or a transitive tournament (there are no cyclic subgraphs), such that these components can be linearly ordered with direction of the edges between the components respecting the linear ordering. In the present paper we extend the notions of irreduciblity and transitivity to tournament kernels and tournament limits, with the intention to prove corresponding decomposition results for touranment kernels and tournament limits.

The rest of this paper is structured as follows. In Section 2 we give the necessary background of directed graph limits. In Section 3 we provide some standard results about tournaments, including a proof of the decomposition result mentioned in the preceding paragraph. In Section 4 we introduce tournament kernels and limits and in Section 5 characterise the transitive tournament kernels and limits. Then, in Section 6, we define direct sums of tournament kernels, and prove a decomposition theorem. In order to extend this to tournament limits, we define in Section 7 what we mean by the direct sum of tournament limits. In Section 8 we show that the notion of irreducibility and transitivity coincides for tournaments, kernels and limits, and in Section 9 we prove a decomposition theorem for tournament limits. We mention that the ideas in Sections 6--9 draw heavily on \cite{Janson2008}.

%% file: master_preliminaries.tex
%%%%%%%%%%%%%%%%%%%%%%%%%%%%%%%%%%%%%%
\section{Preliminaries}
%%%%%%%%%%%%%%%%%%%%%%%%%%%%%%%%%%%%%%
\label{sec:prelim}

A \emph{directed graph} (digraph) $G$ is a pair $(V(G),E(G))$, where $V(G)$ is a countable vertex set and $E(G)\subseteq V(G)\times V(G)$. A \emph{tournament} is a graph $G$ such that for any $i,j\in V(G)$, then $(i,j)\notin E(G)$ if $i=j$, and $(i,j)\in E(G)\Leftrightarrow (j,i)\notin E(G)$ if $i\neq j$. We define two types of random subgraphs for a digraph $G$. Let $G[k]$ be the random (sub--)graph induced by vertices $\{v_1,v_2,\dots, v_k\}\subseteq V(G)$ drawn uniformly at random with replacement, and let $G[k]'$ be the random (sub--)graph induced by vertices $\{v_1,v_2,\dots, v_k\}\subseteq V(G)$ drawn uniformly at random without replacement. 

A map $\Phi:V(F)\to V(G)$ is said to be a \emph{digraph homomorphism} from $F$ to $G$ if $(\Phi(i),\Phi(j))\in E(G)$ for all $(i,j)\in E(F)$, and is said to preserve non--adjacency if also $(\Phi(i),\Phi(j))\notin E(G)$ for all $(i,j)\notin E(F)$. Instead of saying that $\Phi$ is a homomorphism, we will often say that it preserves adjacency. For any digraph $G$, let  $v(G):=|V(G)|$ be the number of vertices. For any two digraphs $F,G$ we define the homomorphism density 
\begin{align}
 t(F,G):=\pP[F\subseteq G[v(F)]]=\frac{\hom (F,G)}{v(G)^{v(F)}}
\end{align}
where $\hom(F,G)$ is the number of digraph homomorphisms $V(F)\to V(G)$. The denominator is the total number of mappings $V(F)\to V(G)$, so it is clear that the final equality holds. In a similar fashion we define the injective and induced homomorphism densities as
\begin{align}
 t_{\text{inj}}(F,G):=\pP[F\subseteq G[v(F)']]=\frac{\text{inj} (F,G)}{(v(G))_{v(F)}} \\
 t_{\text{ind}}(F,G):=\pP[F= G[v(F)']]=\frac{\text{ind} (F,G)}{(v(G))_{v(F)}}.
\end{align}
where $\text{inj} (F,G)$ and $\text{ind} (F,G)$ respectively denote the number of injective homomorphisms $V(F)\to V(G)$ and the number of injective homomorphisms $V(F)\to V(G)$ preserving non--adjacency. The notation $(n)_k$ is the falling factorial defined by $(n)_k=n(n-1)\cdots (n-(k-1))$.

It is well--known that homomorphism numbers and induced homomorphism numbers give the same information in the graph--case. This works in an identical fashion for digraphs as well. Because the only induced subgraphs of tournaments are themselves tournaments, this is particularly useful when considering tournaments. More precisely, we have
\begin{align}
 t_{\text{inj}}(F,G)=\sum_{F'\supseteq F}t_{\text{ind}}(F',G)
\end{align}
where the sum ranges over all digraphs $F'$ with $v(F)= v(F')$ containing $F$ as a subgraph. This formula relates the injective and induced homomorphism numbers. In particular, if $G$ is a tournament the only induced subgraphs are themselves tournaments, so it suffices to sum over tournaments $F'$. Indeed, if $F,G$ are tournaments, then $\text{ind}(F,G)=\text{inj}(F,G)$. Moreover, it is not difficult to show  (e.g. \cite{HellNesetril2004}) that
\begin{align}
 \text{hom}(F,G)=\sum_{\Theta}\text{inj}(F/\Theta,G),
\end{align}
where the sum ranges over all partitions $\Theta$ of $V(F)$ and $F/\Theta$ is the quotient digraph. It is possible to show that 
\begin{align}
 \left|t_{\text{inj}}(F,G)-t(F,G) \right| = O\left(v(G)^{-1} \right)
\end{align}
which implies that the densities $t(\cdot, G), t_{\text{inj}}(\cdot, G)$ and $t_{\text{ind}}(\cdot, G)$ provide the same information for large digraphs $G$.
% 
% We mention also an inclusion--exclusion like formula that relates subgraph densities. Let $F,G$ be fixed tournaments. 
% 
% Choose a subset $E'\subseteq E(F)$ and invert these edges to form the graph $F'$. 
% 
% Let $\cH_i$ be the set of those subgraphs of $F'$ where $i$ edges from $E'$ have been removed, and isolated vertices are ignored. We then have \footnote{Is this for induced or non--induced?}
% \begin{align}
%  t_{ind}(F,G)=\sum_{i=0}^{|E'|}(-1)^{|E'|-i}\sum_{H\in \cH_i}t_{ind}(H,G). \label{eq:relhomdens}
% \end{align}

Let $\cD$ be the set of unlabelled directed graphs. As noted in \cite{DiaconisJanson}, the map $\tau(G):=(t(F,G))_{F\in \cD} \times (v(G)^{-1})\in [0,1]^{\cD}\times [0,1]$ is injective. We may therefore identify $\cD$ with its image $\tau(\cD)$, which is a subset of the compact metrizable space $[0,1]^{\cD}\times [0,1]$. We define $\overline{\cD}$ as the closure of $\cD$ in $[0,1]^{\cD}\times [0,1]$ and the set of digraph limits as $\widehat{\cD}=\overline{\cD}\setminus \cD$. We say that a sequence of digraphs $(G_n)_{n=1}^{\infty}$ \emph{converges} if $v(G_n)\to \infty$ and $t(F,G_n)$ converges for each digraph $F$; or if the sequence $(G_n)_{n=1}^{\infty}$ is eventually constant (up to isomorphism). The set of digraph limits $\widehat{\cD}=\overline{\cD}\setminus \cD$ therefore consists of the limits of sequences $(G_n)_{n=1}^{\infty}$ such that $v(G_n)\to \infty$ and $t(F,G_n)$ converges for each digraph $F$. Similarly, let $\cT\subseteq \cD$ denote the set of all unlabelled tournaments, i.e. those digraphs $G$ with no loops and a single directed edge between every pair of vertices. Like before we can identify $\cT$ with its image $\tau(\cT)$. We call $\widehat{\cT}=\overline{\cT}\setminus \cT$ the set of \emph{tournament limits}.

Throughout we let $(\cS,\mu)$ be some probability space, where we abuse notation and do not mention the underlying $\sigma$--algebra. Without loss of generality we may take $\cS=[0,1]$ and $\mu$ to be Lebesgue measure, but we shall try to be as general as possible here. In the theory of (undirected) graph limits, it is known that each graph limit can be represented by a symmetric function $\cS^2\to [0,1]$. In the directed case, Janson and Diaconis \cite{DiaconisJanson} showed that one should consider a quintuple $\bW=(W_{00},W_{01},W_{10},W_{11},w)$  where $W_{10},W_{01},W_{11}$ and $W_{00}$ are measurable functions $\cS^2\to [0,1]$ such that $W_{10}+W_{01}+W_{11}+W_{00}=1$ a.e., $W_{00}$ and $W_{11}$ are a.e. symmetric, $W_{01}(x,y)=W_{10}(y,x)$ a.e., and $w:\cS \to [0,1]$ is a measurable function. Whenever we refer to a ``quintuple'' we will mean a quintuple like this. Every quintuple  $\bW$ generates a random infinite digraph $G=G(\infty,\bW)$ with vertex set $\N\setminus \{0\}$ as follows. First choose an infinite sequence $X_1,X_2,\dots$ independently and $\mu$--uniformly at random from $\cS$. Then for all each $1\leq i<j$, let the edge probabilities be given by
\begin{align}
 \mathbb{P}[\mathbbm{1}_{(i,j)\in E(G)}=\alpha \text{ and } \mathbbm{1}_{(j,i)\in E(G)}=\beta]=W_{\alpha\beta}(X_i,X_j), \qquad \alpha,\beta\in \{0,1\}.
\end{align}
Finally, put a loop at each vertex $i$ with probability $w(X_i)$, independently of everything else. Using instead a finite sequence $X_1,X_2,\dots, X_n$ allows us to in the same way define a finite random graph $G(n,\bW)$, the distribution of which equals the distribution of the restriction of $G(\infty,\bW)$ to its first $n$ vertices. For a digraph $F$ with $k$ vertices, we define homomorphism and induced homomorphism numbers with respect to the quintuple $\bW$ by
\begin{align}
 t(F,\bW):=\mathbb{P}[F\subseteq G(k,\bW)],
\end{align}
and 
\begin{align}
 t_{\text{ind}}(F,\bW):=\mathbb{P}[F=G(k,\bW)].
\end{align}

The quintuples $\bW$ lead to another type of limit object of convergent graph sequences. We say that a sequence of digraphs $(G_n)_{n=1}^{\infty}$ \emph{converges} to the quintuple $\bW$ if $v(G_n)\to \infty$ and $t(F,G_n)\to t(F,\bW)$ for any finite digraph $F$. If $(G_n)_{n=1}^{\infty}$ converges to the quintuple $\bW$, then it is clear that $(G_n)_{n=1}^{\infty}$ converges also to some digraph limit $\Gamma$ in the space $\overline{\cD}$. If this happens we say that $\Gamma$ is \emph{represented} by $\bW$, and we say that two quintuples $\bW_1$ and $\bW_2$ are \emph{equivalent} if they represent the same graph limit $\Gamma$, i.e. if $t(F,\bW_1)=t(F,\bW_2)$ for all digraphs $F$. While the limit object $\Gamma$ is unique, typically it has many representatives and their relationship is rather complicated; the interested reader may refer to \cite{Lovaszbook} for more details. However, Diaconis and Janson \cite{DiaconisJanson} show that $(G(n,\bW))_{n=1}^{\infty}$ converges to $\bW$, so in fact each $\bW$ gives rise to a digraph limit $\Gamma_{\bW}$.

%% file: master_tourdecomp.tex
%%%%%%%%%%%%%%%%%%%%%%%%%%%%%%%%%%%%%%%%%%%%%%%%%%%%%%%%%%%%%%%%%%%%%%%%%%%%%%%%%%%%%%%%%
\section{Basic facts about tournaments}
%%%%%%%%%%%%%%%%%%%%%%%%%%%%%%%%%%%%%%%%%%%%%%%%%%%%%%%%%%%%%%%%%%%%%%%%%%%%%%%%%%%%%%%%%
In this section we introduce some basic and well--known results for tournaments. We will refer to these facts throughout, and many of the results in the remaining sections will be in direct correspondence with some result in this section.

Throughout we will make use of some special tournaments, which we define now. For all the tournaments presented below, the vertex set is $\{1,2,3,\dots k\}$, so we define the tournaments by giving their edge sets only. Let $\sE_k$ denote the empty graph on $k$ vertices, i.e. $E(\sE_k)=\emptyset$. Let $\sP_k$ denote the path on $k$ vertices, i.e. $E(\sP_k)=\{(1,2),(2,3),\dots, (k-1,k)\}$. Let $\sC_k$ denote the $k$--cycle, i.e. $E(\sC_k)=\{(1,2),(2,3),\dots, (k-1,k) (k,1)\}$. Finally, let $\sT_k$ denote the transitive graph on $k$ vertices, i.e. $E(\sT_k)=\{(i,j)\in \{1,2,\dots, k \}^2 \ : \ i<j \}$. In what follows, by a \emph{countable tournament} we mean a tournament with a finite or countably infinite vertex set.

The following lemma, the proof of which is straightforward, characterises the set of tournaments.

\begin{lemma}\label{lem:whenistournament}
 A digraph $G$ is a tournament if and only if it has no induced subgraphs isomorphic to $\sC_1,\sC_2$ or $\sE_2$.
\end{lemma}

One key notion is that of transitivity. For a countable tournament this is defined as follows.

\begin{definition}
 A countable tournament $G=(V(G),E(G))$ is \emph{transitive} if $(i,k)\in E(G)$ whenever $(i,j)\in E(G)$ and $(j,k)\in E(G)$. It is \emph{acyclic} if it does not contain any cycle $\sC_k$, $k\geq 3$ as a subgraph.
\end{definition}

One of the most well--known results in the study of (finite) tournaments is the following theorem that characterises transitive tournaments. We denote by $(d_i)_{i=1}^{v(G)}$ the number of outgoing edges from the $i$:th vertex of $G$ (in some arbitrary ordering of the vertices).
\begin{theorem}[\cite{Moon1968}]\label{thm:acyclictournament}
Let $G$ be a finite tournament with $n\geq 3$ vertices. Then the following statements are equivalent.
\begin{enumerate}[(1)]
 \item $G$ is transitive. \label{thm:acyclictournament1}
 \item $G$ is acyclic. \label{thm:acyclictournament2}
 \item $G$ contains no cycle of length $3$. \label{thm:acyclictournament3}
\item The vertices of $G$ may be ordered such that $ij\in E(G)$ if and only if $1\leq i<j \leq v(G)$.
 \item $\sum_{i=1}^{n}d_i^2=\frac{n(n-1)(2n-1)}{6}$.
 \item $G$ contains $\binom{n}{k}$ copies of $\sP_k$ as a subgraph, for each $k=1,2,\dots, n$.
\item $G$ contains $\binom{n}{3}$ copies of $\sP_3$ as a subgraph.
  \item $G$ contains $\binom{n}{k}$ copies of $\sT_k$ as a subgraph, , for each $k=1,2,\dots, n$.
\item  $G$ contains $\binom{n}{3}$ copies of $\sT_3$ as a subgraph.
\end{enumerate}
\end{theorem}

Given a set of tournaments, we can define their direct sum as follows.

\begin{definition}\label{def:tourdirsum}
Let $(\cQ,<)$ be a countable and linearly ordered set and let $\{G_i \}_{i\in \cQ}$ be a set of countable tournaments. Define $\bdirsum{G_i}{i\in \cQ}$ to be the countable tournament with vertex set $\bigsqcup_{i\in \cQ}V(G_i)$ and edge set $E(G)=\bigcup_{i\in \cQ}E(G_i) \cup \bigcup_{i<j}\{(v,w) \ : \ v\in V(G_i), w\in V(G_j)\}$. As a special case, define $G_1 \dirsum{}{} G_2$ as $\bdirsum{G_i}{i\in \{1,2\}}$.
\end{definition}

The set of irreducible tournaments consists of those tournaments which do not admit a decomposition into a direct sum of subtournaments.

\begin{definition}
 A countable tournament $G=(V(G),E(G))$ is \emph{reducible} if there exists vertex--disjoint induced subtournaments $G_1,G_2$ such that $V(G_1)\cup V(G_2)=V(G)$ and $G=G_1 \dirsum{}{} G_2$. Otherwise $G$ is \emph{irreducible}.
\end{definition}

Related to the notion of irreducibility is the notion of strong connectedness. Given a tournament $G=(V(G),E(G))$, we define the outneighbourhood of $X\subseteq V(G)$ by 
\begin{align}
N(X)=\{w\in V(G) \ : \ (v,w)\in E(G) \text{ for some }v\in X \}.
\end{align}
This can be extended by defining $N^{0}(X)=X$ and $N^m(X)=N(N^{m-1}(X))$ for all $m\geq 1$. Note that $N^1(X)=N(X)$. We say that a countable tournament is \emph{strongly connected} if and only if for all $v\in V(G)$, we have $V(G)=\bigcup_{m=0}^{\infty}N^m(\{v\})$. For finite tournaments one may bound the indices of the union from above by $v(G)$. 

Moon \cite{Moon1968} gives the following theorem for finite tournaments, which we state here for countably tournaments. The extension is straightforward and we omit the proof. (For instance, one might show the implications $\ref{tour1}\Rightarrow \ref{tour4} \Rightarrow \ref{tour2} \Rightarrow \ref{tour1}$ and $\ref{tour2} \Leftrightarrow \ref{tour3}$.)

\begin{theorem}[\cite{Moon1968}]\label{thm:irreducibletour}
Let $G$ be a countable tournament. The following statements are equivalent.
\begin{enumerate}[(1)]
  \item $G$ is irreducible. \label{tour1}
  \item $G$ is strongly connected. \label{tour2}
  \item For any $v,w\in V(G)$ there exist (finite) directed paths from $v$ to $w$ and from $w$ to $v$. \label{tour3} \label{thm:irreducibletour3}
  \item There does not exist a proper subset $X\subseteq V(G)$ for which $N(X)\subseteq X$. \label{tour4}
\end{enumerate}
\end{theorem}

We shall use the third notion of Theorem \ref{thm:irreducibletour} to prove a decomposition theorem for countable tournaments, up to order isomorphism of the labels of the components. Define an equivalence relation $\sim$ on $V(G)$ by $x\sim y$ if and only if $x=y$ or there are finite directed paths from $x$ to $y$ and from $y$ to $x$. This is an equivalence relation which partitions $V(G)$ into a set of equivalence classes. Each equivalence class $[x]$ in turn induces a subtournament of $G$. 

We claim that the equivalence classes can be ordered linearly in a natural way. Suppose that $x,y,z\in V(G)$ lie in different equivalence classes. If $(x,y),(y,z),(z,x)\in E(G)$, then $x\sim y$ (since there are directed paths in both directions). But we assumed that $[x]\neq [y]$, so this is impossible. Therefore, for any three vertices $x,y,z\in V(G)$ in different equivalence classes with $(x,y),(y,z)\in E(G)$, we have $(x,z)\in E(G)$. For a similar reason, if $x_1\sim x_2$ and $y_1\sim y_2$ but $x_1,x_2\not\sim y_1,y_2$, then $(x_1,y_1)\in E(G) \Leftrightarrow (x_2,y_2)\in E(G)$. This results in a linear ordering of the equivalence classes. Since $V(G)$ is a countable set, the number of equivalence classes is countable. Therefore the ordering of the equivalence classes is order isomorphic to some ordered set $\cQ$, where we may label the equivalence classes by the corresponding element in $\cQ$.

By Theorem \ref{thm:irreducibletour}, each equivalence class induces an irreducible subtournament of $G$. With this notation in mind, the above construction gives us the existence of a decomposition of countable tournaments. We prove uniqueness below, noting that our main goal is to later extend the following theorem to the tournament limits and tournament kernels.

\begin{theorem}\label{thm:tourdecomp1}
Let $G=(V(G),E(G))$ be a countable tournament. Then there exists some countable ordered set $\cQ$ and disjoint subtournaments $\{G_i \}_{i\in \cQ}$ such that
\begin{align}
 G=\bdirsum{G_i}{i\in \cQ},
\end{align}
where each $G_i$ is irreducible. Furthermore, this decomposition is unique up to order isomorphism of $\cQ$.
\end{theorem}

\begin{proof}
%  Existence follows from the construction above; we prove uniqueness now. Let $G$ be a countable graph and  
% \begin{align}
%  \bdirsum{G_i}{i\in \cQ}= G=\bdirsum{G_i'}{i\in \cQ'}
% \end{align}
% where $\cQ,\cQ'$ are linearly ordered sets with $\cQ=\cQ_1\cup \cQ_2\cup \cI$ and $\cQ'=\cQ'_1\cup \cQ'_2\cup \c'I$, and $\cQ_1,\cQ_2,\cI$ and $\cQ'_1, \cQ'_2, \cI'$ are defined as in the statement of the theorem. We claim that the triples $(\cQ_1,\cQ_2,\cI)$ and $(\cQ'_1,\cQ'_2,\cI')$ are order isomorphic. Let us first show that $\cQ'$ and $\cQ'$ are order isomorphic, and then 
Existence follows from the construction above; we prove uniqueness here. Suppose the two graphs
\begin{align}
 G& = \bdirsum{G_i}{i\in \cQ} \\
 G'&= \bdirsum{G'_i}{i\in \cQ'}
\end{align}
are isomorphic, where $\cQ,\cQ'$ are linearly ordered sets and each $G_i,G_j'$ is irreducible. Since $G$ and $G'$ are isomorphic, there exists a bijection $f:V(G)\to V(G')$ that preserves edges and non--edges (equivalently, preserving the direction of edges between all pairs of vertices). It suffices to prove that this bijection also preserves the ordering of the sets $\cQ$ and $\cQ'$. Denote by $f(G_i)$ the induced subtournament of $G'$ with vertex set $\{f(u) \ : \ u\in V(G_i) \}$. This is isomorphic to $G_i$, so $f(G_i)$ must be irreducible (since $G_i$ is). This implies there is exactly one $j\in \cQ'$ such that $G'_j\supseteq f(G_i)$. We want to show equality here. Let $\cQ_j=\{i\in \cQ \ : \ f(G_i)\subseteq G'_j\}$. Assume, for contradiction, that $| \cQ_j | > 1$. Since $f$ preserves adjacency and non--adjacency, we have $G_j'=f\left( \bdirsum{G_i}{i\in \cQ_j}\right)=\bdirsum{f(G_i)}{i\in \cQ_j}$, contradicting the irreducibility of $G_j'$.

Note that $|\cQ_j|>0$ for every $j\in \cQ'$ (by surjectivity of $f$), so this implies that $|\cQ_j|=1$ for every $j\in \cQ'$. This shows that $f$ induces a bijection $\cQ\to \cQ'$. We need to show that this is order--preserving. Let $i_0<i_1$ in $\cQ$ and let $j_0$ and $j_1$ be the unique elements of $\cQ'$ such that $f(G_{i_0})=G'_{j_0}$ and $f(G_{i_1})=G'_{j_1}$. For any $u\in V(G_{i_0})$, $v\in V(G_{i_1})$, we have $(f(u),f(v))\in E(G_2)$. This implies that $i_0 < i_1$ (in $\cQ$) if and only if $j_0<j_1$ (in $\cQ'$), so $\cQ$ and $\cQ'$ are order isomorphic.
% It is not difficult to show that if there exists $i\in \cQ$ and $j\in \cQ'$ such that $\emptyset \neq G_i\cap G_j' \neq G_i $ or  $\emptyset \neq G_i\cap G_j' \neq G_j'$, then either $G_i$ or $G_j'$ is reducible. For example, if $\emptyset \neq G_i\cap G_j' \neq G_i $ \emph{and}  $\emptyset \neq G_i\cap G_j' \neq G_j$, one obtains that $G_i=(G_i\setminus G_j) \dirsum{}{} (G_i\cap G_j) $ or $G_i= (G_i\cap G_j)\dirsum{}{}  (G_i\setminus G_j)$, contradicting the irreducibility of $G_i$. Similar contradictions occur if only one of the inequalities hold. Hence, for each $i\in \cQ$ there exists some $j\in \cQ'$ such that $G_i=G_j'$, which implies that there is a bijection $f:\cQ\to \cQ'$ such that $G_{f(i)}=G_i'$, $i\in \cQ$. 
% 
% To show that $f$ is order--preserving, suppose $i_0,i_1\in \cQ$ with $i_0<i_1$ and $j_0,j_1\in \cQ'$ with $f(i_0)=j_0$ and $f(i_1)=j_1$. Suppose for contradiction that $j_0>j_1$. Since $i_0<i_1$ we have $G_{i_0}\to G_{i_1}$. Since $j_0>j_1$ we have $G_{i_1}=G_{j_1}'\to G_{j_0}'=G_{i_0}$, which is a contradiction. Hence $j_0<j_1$, so $f$ is an increasing bijection. Uniqueness follows.
\end{proof}

Some of the irreducible subtournaments in Theorem \ref{thm:tourdecomp1} are singletons; and some of these singletons appear consecutively in the ordering $\cQ$ (by consecutively we mean that there is no non--singleton irreducible subtournament between these singletons). Now merge the singletons of the decomposition in Theorem \ref{thm:tourdecomp1} whose labels appear consecutively in the ordering $\cQ$. Each such merged group of vertices induces a transitive subtournament. For each merged class, choose its label arbitrarily from the labels of its constituent equivalence classes. This merging procedure can be made more formal, but we hope the idea is clear. This gives us the following corollary.

\begin{corollary}\label{cor:tourdecomp}
Let $G=(V(G),E(G))$ be a countable tournament. Then there exists some countable ordered set $\cQ$ with disjoint subsets $\cQ_1,\cQ_2,\cI$ such that $\cQ=\cQ_1\cup \cQ_2\cup \cI$ and disjoint subtournaments $\{G_i \}_{i\in \cQ}$ such that
\begin{align}
 G=\bdirsum{G_i}{i\in \cQ},
\end{align}
where each $G_i$, $i\in \cQ_1$, is irreducible and has at least two vertices, each $G_i$, $i\in \cQ_2$, is transitive and has at least two vertices, and each $G_i$, $i\in \cI$ is a singleton. Also, for any $i<j$ in $\cQ_2\cup \cI$, there exists $k\in \cQ_1$ with $i<k<j$. Furthermore, this decomposition is unique up to order isomorphism of $\cQ$.
%\footnote{By this we mean that for any other decomposition $G=\bdirsum{G_i'}{i\in \cQ'}$ with $\cQ'=\cQ_1'\cup \cQ_2' \cup \cI'$ there is an order--preserving bijection $\Phi:\cQ \to \cQ'$ such that $\Phi(\cQ_1)=\cQ_1', \Phi(\cQ_2)=\cQ_2'$ and $\Phi(\cI)=\cI'$.}
\end{corollary}

Our final result in this section is a formula for the induced homomorphism densities of a direct sum of tournaments. We will need this later when studying the direct sums of tournament limits.

Given a finite tournament $F$ and a countable ordered set $\cQ$, let $\cP(F,\cQ)$ denote the set of decompositions $\bdirsum{F_i}{i\in \cQ}$ of $F$, where all but finitely many of the $F_i$ are non--empty. Note that any such decomposition can be obtained by merging consecutive elements of the decomposition of $F$ given in Theorem \ref{thm:tourdecomp1}.

\begin{theorem}\label{lem:tourdens}
 Let $\cQ=\{1,\dots, p\}$, and let $(G_i)_{i\in \cQ}$ be a sequence of finite tournaments. Then, for any finite tournament $F$,  
\begin{align}
 t_{\text{ind}}\left(F,\bdirsum{G_i}{i\in \cQ}\right)=\frac{1}{v(G)_{v(F)}}\sum_{\cP(F,\cQ)}\prod_{i\in \cQ}(v(G_{i}))_{v(F_i)}t_{\text{ind}}(F_i,G_{i}).
\end{align}
\end{theorem}

\begin{proof}
Let $G=\bdirsum{G_i}{i\in \cQ}$. We compute the number of injective maps $V(F)\to V(G)$ that preserve adjacency and non--adjacency, and then do the appropriate normalisations to get the induced homomorphism density. 

First, note that each map $\Phi:V(F)\to V(G)=\bigsqcup_{i=1}^p V(G_i)$ induces an ordered partition $\{\Phi^{-1}(V(G_{1})),\dots, \Phi^{-1}(V(G_{p})) \}$ of $V(F)$. Moreover $F$ is isomorphic to $\bdirsum{ \Phi^{-1}(G_{i})}{i\in \cQ}$. 

Fix a decomposition $F=\bdirsum{F_i}{i\in \cQ}$, where some of the $F_i$ may be empty. We count the number of injective maps $\Phi:V(F)\to V(G)$ that preserve adjacency and non--adjacency and induce the fixed decomposition. First, given $\Phi:V(F)\to V(G)$, this can only happen if $\Phi^{-1}(V(G_{i}))\supseteq V(F_i)$ for each $i\in \cQ$. Given this, the restrictions $\Phi|_{V(F_i)}:V(F_i)\to V(G_i)$, for all $i\in \cQ$, must preserve adjacency and non--adjacency. For each $i\in \cQ$ there are $\text{ind}(F_i,G_i)$ such maps. Therefore there are precisely $\prod_{i\in \cQ}\text{ind}(F_i,G_{i})$ injective maps that preserve adjacency and non--adjacency and give rise to the given decomposition. Summing over all decompositions of $F$ we find
\begin{align}
 \text{ind}(F,G)=\sum_{\cP(F,\cQ)}\prod_{i\in \cQ} \text{ind}(F_i,G_{i}).
\end{align}
The result follows by using $v(F)=\sum_{i\in \cQ}v(F_i)$ and the defining relation
\begin{align}
 t_{\text{ind}}(F,G)=\frac{\text{ind}(F,G)}{(v(G))_{v(F)}}.
\end{align}
\end{proof} 

%% file: master_kernels.tex
\section{Preliminary results on tournament limits and kernels} \label{sec:graphlimits}
%%%%%%%%%%%%%%%%%%%%%%%%%%%%%%%%%%%%%%%%%%%%%%
In this section we begin our study of tournament limits and kernels. First, the following theorem characterises the set of tournament limits. The equivalence between \ref{thm:equiv1} and \ref{thm:equiv3} is motivated by Lemma \ref{lem:whenistournament}. The proof that follows is similar to that of Theorem 10.1 in Janson \cite{Janson2011b}, the proof of which is omitted and can be found in the preprint \cite{Janson2011bpreprint}. 

\begin{theorem} \label{thm:equiv}
 Let $\Gamma\in \widehat{\cD}$ be a digraph limit. Then the following are equivalent.
\begin{enumerate}[(i)]
 \item $\Gamma\in \widehat{\cT}$, i.e. $\Gamma$ is a tournament limit. \label{thm:equiv1}
 \item $t_{\text{ind}}(F,\Gamma)=0$ whenever $F$ is a finite digraph that is not a tournament. \label{thm:equiv2}
 \item $t_{\text{ind}}(\sC_1,\Gamma)=t_{\text{ind}}(\sC_2,\Gamma)=t_{\text{ind}}(\sE_2,\Gamma)=0$.\label{thm:equiv3}
 \item If $\bW=(W_{10},W_{01},W_{11},W_{00},w)$ represents $\Gamma$, then $w=0=W_{11}=W_{00}$ a.e. and $W_{10}(x,y)+W_{10}(y,x)=1$ a.e..\label{thm:equiv4}
\end{enumerate}
\end{theorem}

\begin{proof}
\ref{thm:equiv1} $\Longrightarrow$ \ref{thm:equiv2}. Every induced subgraph of a tournament is itself a tournament, so $t_{\text{ind}}(F,G)=0$ whenever $G$ is a tournament and $F$ is a digraph that is not a tournament. By continuity also $t_{\text{ind}}(F,\Gamma)=0$ for all $\Gamma\in \widehat{\cT}$.
  
\ref{thm:equiv2} $\Longrightarrow$ \ref{thm:equiv3}. The digraphs $\sC_1,\sC_2,\sE_2$ are not tournaments.
  
\ref{thm:equiv3} $\Longrightarrow$ \ref{thm:equiv4}. Suppose $\Gamma$ has a representative $\bW=(W_{10},W_{01},W_{11},W_{00},w)$. It suffices to show $W_{11}=W_{00}=0$ a.e. and $w=0$ a.e. We have
\begin{align}
 0=t_{ind}(\sC_1,\Gamma)=\pP[\sC_1= G(1,\bW)]=\pP[(1,1)\in E(G(1,\bW))]=\E[w(X_1)],
\end{align}
so $w=0$ a.e.. Similarly, given that $w=0$,
\begin{align}
 0=t_{ind}(\sC_2,\Gamma)=\pP[\sC_2 = G(2,\bW)] &=\pP[(1,2),(2,1)\in E(G(2,\bW)),(1,1),(2,2)\notin E(G(2,\bW))] \\
 &=\pP[(1,2),(2,1)\in E(G(2,\bW))] \\
  &=\E[W_{11}(X_1,X_2)] 
\end{align}
and
\begin{align}
  0=t_{ind}(\sE_2,\Gamma)=\pP[\sE_2 = G(2,\bW)] 
 &=\pP[(1,2),(2,1),(1,1),(2,2)\notin E(G(2,\bW))] \\
 &=\pP[(1,2),(2,1)\notin E(G(2,\bW))] \\
 &=\E[W_{00}(X_1,X_2)],
\end{align}
so $W_{00}=W_{11}=0$ a.e.. By the characterisation mentioned in Section \ref{sec:prelim}, it follows that $W_{10}(x,y)+W_{10}(y,x)=1$ a.e..
 
\ref{thm:equiv4} $\Longrightarrow$ \ref{thm:equiv1}. It follows from the construction of $G(k,\bW)$ in Section \ref{sec:prelim} that $G(k,\bW)$ does not contain $\sC_1,\sC_2$ or $\E_2$ as an induced subgraph (almost surely), so by Lemma \ref{lem:whenistournament}, the random digraphs $G(k,\bW)$ are tournaments almost surely. But $G(k,\bW)$ converges to $\Gamma$ as $k\to \infty$, so $\Gamma \in \widehat{\cT}$.
  \end{proof}

Hence, any quintuplet $\bW$, representing a tournament limit $\Gamma\in \widehat{\cT}$, may be identified with its first coordinate $W:=W_{10}$. This motivates the following terminology.
\begin{definition}
 A \emph{tournament kernel} is a measurable function $W:\cS^2\to[0,1]$ such that $W(x,y)+W(y,x)=1$ for $\mu\times\mu$--a.e. $(x,y)\in \cS^2$.
\end{definition}
Henceforth we will make the identification $\bW=W_{10}=W$, and we will say that $W$ represents $\Gamma$ if $\bW$ represents $\Gamma$, etc.

\begin{lemma}\label{lem:intform}
Let $W:(\cS^2,\mu)\to [0,1]$ be a tournament kernel. This defines a tournament limit $\Gamma_W \in \widehat{\cT}$ such that $G(n,W)\to \Gamma_{W}$ a.s. as $n\to \infty$ and
\begin{align}
 t(F,\Gamma_W)&=\int_{\cS^{v(F)}}\prod_{(i,j)\in E(F)}W(x_i,x_j)\prod_{i=1}^{v(F)}d\mu(x_i) \\
 t_{ind}(F,\Gamma_W)&=\int_{\cS^{v(F)}}\prod_{(i,j)\in E(F)}W(x_i,x_j)\prod_{i=1}^{v(F)}d\mu(x_i).
\end{align}
for all finite tournaments $F$.
\end{lemma}

\begin{proof}
The fact that $W$ defines a tournament limit $\Gamma_{W}$ follows from \cite{DiaconisJanson}. It is not difficult to see that the formulae hold, e.g. 
\begin{align}
 t_{ind}(F,\Gamma_W) = \pP[F=G(v(F),W)] 
&= \E\left[\mathbbm{1}_{\{F=G(v(F),W)\}}\right] \\
&= \E\left[\prod_{(i,j)\in E(F)} \mathbbm{1}_{\{(i,j)\in E(G(v(F),W))\}} \right] \\
&=\int_{\cS^{v(F)}} \prod_{(i,j)\in E(F)}W(x_i,x_j)\prod_{i=1}^{v(F)}d\mu(x_i).
\end{align}
\end{proof}

%% file: master_transitive.tex
%%%%%%%%%%%%%%%%%%%%%%%%%%%%%%%%%%%%%%%%%%%%%%%
\section{Transitive kernels}
%%%%%%%%%%%%%%%%%%%%%%%%%%%%%%%%%%%%%%%%%%%%%%
In this section we obtain an analogue of Theorem \ref{thm:acyclictournament} in the setting of tournament kernels. We say that a tournament kernel $W$ is \emph{acyclic} if $t(C_k,W)=0$ for all $k\geq 3$, and that it is \emph{transitive} if $W(x,y),W(y,z)>0 \Rightarrow W(x,z)=1$ for $\mu^3$--a.e. $(x,y,z)\in \cS^3$.

We need the following standard viewpoint of the adjacency matrix of a directed graph. A tournament $G=(V(G),E(G))$ with $v(G)=n$ vertices labeled $1,\dots, n$ defines an \emph{adjacency function} 
\begin{align}
 A_G(x,y) = \begin{cases} 1 & (i,j)\in E(G), x\in \left(\frac{i-1}{n},\frac{i}{n} \right], y\in \left( \frac{j-1}{n},\frac{j}{n} \right], \\
             0 & \text{otherwise}.
            \end{cases} \label{eq:graphgraphon}
\end{align}
Note that $A_G$ is almost a tournament kernel, in the sense that $A_G(x,y)+A_G(y,x)=1$ except on the set $\bigcup_{i=1}^n\left(\frac{i-1}{n},\frac{i}{n} \right]^2$, the measure of which tends to zero as $n\to \infty$. Note that 
\begin{align}
 t(F,G)=\int_{[0,1]^{v(F)}}\prod_{ij\in E(F)}A_G(x_i,x_j)\prod_{i\in V(F)} d\mu(x_i),
\end{align}
which should be compared to the expressions in Lemma \ref{lem:intform}.

\begin{lemma}\label{lem:acycliconv}
 The sequence $(\sT_n)_{n=1}^{\infty}$ of transitive tournaments converges to the tournament limit represented by $W_T:([0,1]^2,\Leb)\to [0,1]$, where $\Leb$ is Lebesgue measure and $W_T(x,y)=\mathbbm{1}_{\{x\leq y\}}$.
\label{lem:conv}
\end{lemma}

\begin{proof}
By Theorem \ref{thm:equiv} each tournament limit $\Gamma$ can be represented by a tournament kernel $W:=W_{1,0}$. Let $A_{n}:=A_{\sT_n}$ be the function defined as in \eqref{eq:graphgraphon}, i.e. 
\begin{align}
 A_n(x,y)=\mathbbm{1}_{\{x\leq \frac{\lceil ny \rceil}{n}\}}.
\end{align}

Let $F=(V(F),E(F))$ be a finite digraph. Then 
\begin{align}
 t(F,G_n)
&=\int_{[0,1]^{v(F)}}\prod_{ij\in E(F)}A_n(x_i,x_j) \prod_{i\in V(F)}d\Leb(x_i)\\
&\to  \int_{[0,1]^{v(F)}}\prod_{ij\in E(F)}\mathbbm{1}_{\{x_i\leq x_j\}} \prod_{i\in V(F)}d\Leb(x_i) \\
& = t(F,W_T)
\end{align}
by the dominated convergence theorem and the fact that $A_n$ converges pointwise a.e. to $W_T$. Hence $(\sT_n)_{n=1}^{\infty}$ converges to the tournament limit that is represented by $W_T$.
\end{proof}

\begin{figure}
\centering
\includegraphics[width=0.9\textwidth]{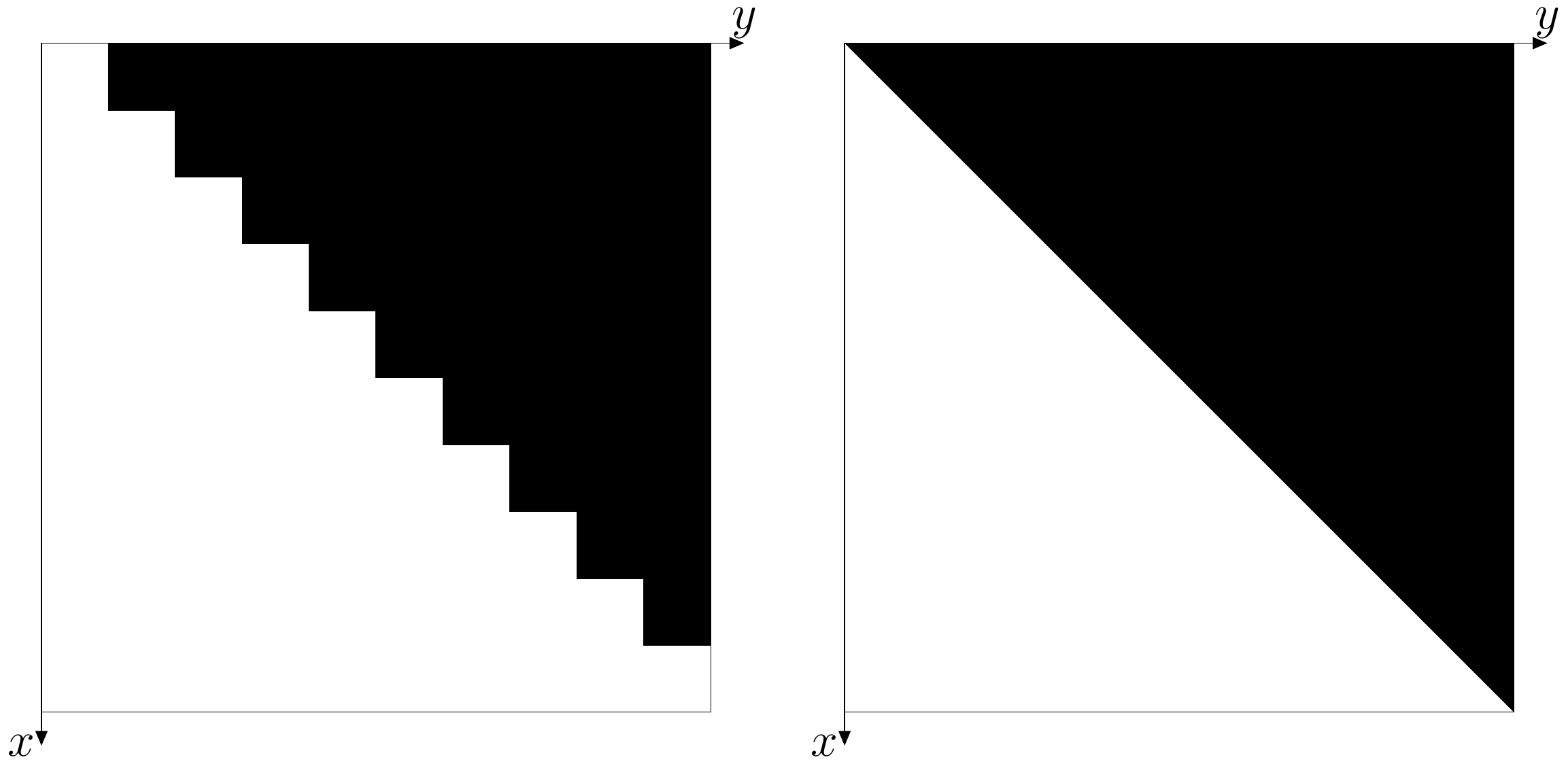}
\caption{The leftmost figure shows the function $A_{10}(x,y)$, which is $1$ on the black region and $0$ on the white region. The rightmost figure shows the transitive kernel $W:([0,1]^2, \lambda)\to [0,1]$, defined by $W(x,y)=\mathbbm{1}_{\{x\leq y\}}$.}
\end{figure} 

We say that a family of tournaments $\cP\subseteq \cT$ is \emph{hereditary} if for all $G\in \mathcal{P}$, then $F\in \cP$ whenever $F$ is a subtournament of $G$. Janson \cite{Janson2011} gives the following theorem about hereditary properties of graph limits, which can be adopted to digraph limits or in this case tournament limits. In that setting one requires $F$ to be an induced subgraph, but in the tournament setting, the subtournaments are precisely the induced subgraphs.

\begin{theorem}[Janson \cite{Janson2011}]
 Let $\cP\subseteq \cT$ be hereditary and let $\Gamma\in \widehat{\cT}$ be represented by the tournament kernel $W$. Then $\Gamma\in \widehat{\cP}=\overline{\cP}\setminus \cP$ if and only if $G(n,W)\in \cP$ a.s. for every $n\geq 1$.
\label{thm:hereditary}
\end{theorem}

The family $\cP=\{G\in \cT \ : \ t(\sC_3,G)=0 \}$ of transitive tournaments is an example of a hereditary family. 

\begin{lemma}\label{lem:weakiso}
Let $\Gamma\in \widehat{\cT}$ be a tournament limit. If $t(\sC_3,\Gamma)=0$, then $\Gamma$ can be represented by the tournament kernel $W_T(x,y)=\mathbbm{1}_{\{x\leq y\}}$.
\end{lemma}
\begin{proof}
Suppose that $\Gamma$ is representable by the tournament kernel $W$ and recall the construction of the random graphs $G(n,W)$ in Section \ref{sec:graphlimits}. As mentioned, Diaconis and Janson \cite{DiaconisJanson} show that $G(n,W)\to W$ almost surely; see also \cite{Boeckner}. Since the family of tournament $\cP=\{G\in \cT \ : \ t(\sC_3,G)=0 \}$ of transitive tournaments is hereditary, by Theorem \ref{thm:hereditary} we know that $t(\sC_3,G(n,W))=0$ a.s. Therefore,  Lemma \ref{lem:conv} implies that $G(n,W)\to \mathbbm{1}_{\{x\leq y\}}$. But this means that $W$ is equivalent to $W_T(x,y)=\mathbbm{1}_{\{x\leq y\}}$, so $\Gamma$ can be represented by $W_T$.
\end{proof}

Inspired by Theorem \ref{thm:acyclictournament} we prove the following theorem that completely characterises transitive tournament kernels.

\begin{theorem} \label{thm:acyclic}
 Let $W:(\cS^2,\mu)\to [0,1]$ be a tournament kernel. Then the following statements are equivalent.
\begin{enumerate}[(1)]
 \item $W$ is transitive. \label{thm:acyclic1}
 \item $W$ is acyclic, i.e. $t(\sC_k,W)=0$ for all $k\geq 3$. \label{thm:acyclic2}
 \item $t(\sC_3,W)=0$. \label{thm:acyclic3}
 \item $W$ is equivalent to $W_T:([0,1]^2, \Leb)\to [0,1]$ given by $W_T(x,y)=\mathbbm{1}_{\{x\leq y\}}$.
 \label{thm:acyclic4} 
  \item $\int_{\cS}\left( \int_{\cS} W(x,y)d\mu(y)\right)^2 d\mu(x)=\frac{1}{3}$. \label{thm:acyclic5}
  \item For all $k\geq 1$, $t(\sP_k,W)=\frac{1}{k!}$. \label{thm:acyclic6}
  \item $t(\sP_3,W)=1/6$. \label{thm:acyclicP3}
  \item For all $k\geq 1$, $t(\sT_k,W)=\frac{1}{k!}$. \label{thm:acyclic7}
  \item $t(\sT_3,W)=1/6$. \label{thm:acyclicT3}
%  \item For any $B\subseteq \cS^2$, it holds that $\int_{\cS\setminus B\times \cS}W(x,y)d\mu(y)=0$\todo{... something like this? for later proof?}
\end{enumerate}
\end{theorem}

\begin{proof}

It is trivial to show \ref{thm:acyclic2} $\Longrightarrow$ \ref{thm:acyclic3} $\Longleftrightarrow$ \ref{thm:acyclic1} and \ref{thm:acyclic6} $\Longrightarrow$ \ref{thm:acyclicP3} and \ref{thm:acyclic7} $\Longrightarrow$ \ref{thm:acyclicT3}.

\ref{thm:acyclic3} $\Longrightarrow$ \ref{thm:acyclic4}. This follows from the proof of Lemma \ref{lem:weakiso}. 

\ref{thm:acyclic4} $\Longrightarrow$ \ref{thm:acyclic2}. This follows from the fact that 
\begin{align}
\int_{[0,1]^k} \prod_{i=1}^k \mathbbm{1}_{\{x_i\leq x_{i+1}\}}\prod_{i=1}^k d\Leb(x_i)=0,
\end{align} 
where we make the identification $x_{k+1}=x_1$.

\ref{thm:acyclic4} $\Longrightarrow$ \ref{thm:acyclic5}, \ref{thm:acyclic6}, \ref{thm:acyclic7}. Easy to check by a direct computation.

\ref{thm:acyclic5} $\Longleftrightarrow$ \ref{thm:acyclicP3}. This follows by the direct computation 
 \begin{align}
\frac{1}{3}
&=\int_{\cS}\left( \int_{\cS} W(x,y)d\mu(y)\right)^2 d\mu(x) \\
&= \int_{\cS}\left( \int_{\cS} W(x,y)d\mu(y) \int_{\cS}W(x,z) d\mu(z) \right) d\mu(x) \\
&=\int_{\cS}\left( \int_{\cS} W(x,y)d\mu(y) \int_{\cS}(1-W(z,x)) d\mu(z) \right) d\mu(x) \\
&=\int_{\cS^2} W(x,y) d\mu(y)d\mu(x)  - \int_{\cS^3}W(z,x)W(x,y) d\mu(z) d\mu (y) d\mu(x) \\
&=\frac{1}{2}-t(\sP_3,W).
\end{align}

% (\ref{thm:acyclic4}) $\Rightarrow$ (\ref{thm:acyclic1}). Observe that $T(x,y),T(y,z)>0 \Rightarrow x\geq y \geq z$, so $T(x,z)=1$. \todo{Even after a measure--preserving transformation?!}

\ref{thm:acyclic3} $\Longleftrightarrow$ \ref{thm:acyclicP3} $\Longleftrightarrow$ \ref{thm:acyclicT3}. By using the relation $W(x,z)+W(z,x)=1$ which holds almost everywhere, one obtains
\begin{align}
 t(\sC_3,W)
&=\int_{\cS^3}W(x,y)W(y,z)W(z,x)d\mu(x) d\mu(y) d\mu(z) 
\\
&= \int_{\cS^3}W(x,y)W(y,z)d\mu(x) d\mu(y) d\mu(z)  \\
& \qquad - \int_{\cS^3}W(x,y)W(y,z)W(x,z)d\mu(x) d\mu(y) d\mu(z)  \\
&=t(\sP_3,W)-t(\sT_3,W).
\end{align}
Similarly one shows that $t(\sC_3,W)=-\frac{1}{4}+\frac{3}{2}t(\sP_3,W)$. Hence the three quantities $t(\sC_3,W)$, $t(\sT_3,W)$ and $t(\sP_3,W)$ are uniquely determined by each other, and the equivalence follows.
\end{proof}

% 
% 
% By the inclusion--exclusion like formula \eqref{eq:relhomdens} in Section 2 we have that $t(\sC_3,W)=1-3t(\sP_2,W)+3t(\sP_3,W)-t(\sC_3)$, which implies that $t(\sC_3,W)=\frac{3}{2}t(\sT_3,W)-\frac{1}{6}$ since $3t(\sP_2,W)=\frac{1}{2}$ for all kernels $W$. But then $t(\sP_3,W)=\frac{1}{6}$ implies that $t(\sC_3,W)=0$.
% 
% \ref{thm:acyclicT3} $\Longrightarrow$ \ref{thm:acyclic3}. It follows from \eqref{eq:relhomdens} that $t(\sC_3,W)=\frac{1}{2}-3t(\sT_3,W)$, which implies that $t(\sC_3,W)=0$ if $t(\sT_3,W)=\frac{1}{6}$.

The transitive kernels are important special cases of several classes of tournament kernels. It can be shown that any transitive tournament kernel necessarily is $0-1$ valued. Such kernels have have been studied in the undirected case and been been called \emph{random--free}, see e.g. \cite{HatamiNorine2013, LovaszSzegedy2010}. This name stems from the fact that the random graphs $G(n,W)$ depend only on the choice of random points $X_1,X_2,\dots, X_n$ without further randomness.

Transitive kernels are also special cases of so--called directed threshold graphs. These are digraphs $G$, with a vertex weight function $f:\{1,\dots, v(G)\}^2 \to \R^+$, for which $(i,j)\in E(G)$ if and only if $f(i)<f(j)$ and $f(i)+f(j)>t$ , where $t$ is some fixed constant. Choosing $t=0$ and $f(i)=i$ for each $i\in \{1,\dots, v(G)\}$ we retrieve the transitive tournaments as a special case of directed threshold graphs. Limits of directed threshold graphs where studied in \cite{Boeckner}, as an extension to the paper \cite{DiaconisHolmesJanson} that considered the limits of (undirected) threshold graphs.

The transitive kernel is also an example of a \emph{finitely forcible} kernel. These are kernels which are determined (up to equivalence) by specifying a finite number of homomorphism numbers, i.e. specifying that $t(F_i,W)=a_i$ for some finite set $\{(a_i,F_i) \}_{i=1}^n\subseteq \R\times \cD$. We have shown in Theorem \ref{thm:acyclic} that $\{(0,\sC_3)\}$, $\{(1/6,\sP_3)\}$ and $\{(1/6,\sT_3)\}$ are finite forcing families for the transitive kernel. Undirected finitely forcible kernels were studied by Lovász and Szegedy \cite{LovaszSzegedy2011}, and it is possible that some of their results extend to the setting of digraph limits or tournament limits.

%% file: master_kerneldecomp4.tex
%%%%%%%%%%%%%%%%%%%%%%%%%%%%%%%%%%%%%%%%%%%%%%%%
\section{Direct sums and decompositions of tournament kernels}
%%%%%%%%%%%%%%%%%%%%%%%%%%%%%%%%%%%%%%%%%%%%%%%%
\label{sec:kerneldecomp}

A \emph{direct sum} of tournament kernels consists of the following objects.
\begin{itemize}
 \item a countable set $\cQ$,
 \item a non--atomic probability space $(\cI,\mu_{\cI})$,
 \item non--negative constants $(\alpha_i)_{i\in \cQ}$ such that $\sum_{i\in \cQ}\alpha_i \leq 1$,
 \item tournament kernels $(W_i:(\cS_i,\mu_i)\to [0,1])_{i\in \cQ}$, and 
 \item an map $\eta:\cQ \cup \cI \to [0,1]$ such that $\eta|{\cQ}$ is injective and $\eta|_{\cI}$ is injective up to $\mu_{\cI}$--null sets.
\end{itemize}
The direct sum consisting of these objects is denoted $\left(\bdirsum{\alpha_iW_i}{i\in \cQ},\cI,\eta \right)$, and it is equivalent to the kernel $W:(\cS^2,\mu)\to [0,1]$, where
\begin{align}
 \cS=\cI \sqcup \left(\bigsqcup_{i\in \cQ}\cS_i \right),
\end{align}
and
\begin{align}
\mu=\sum_{i\in\cQ}\alpha_i\mu_i+\left(1-\sum_{i\in\cQ} \alpha_i \right)\mu_{\cI},
\end{align} 
defined by 
\begin{align}
 W(x,y) = 
\begin{cases} W_i(x,y) , &x,y\in \cS_i, \ i \in \cQ, \\
0, & \eta(x)<\eta(y),  \\
1, & \eta(x)>\eta(y).
\end{cases}
\end{align}
where we have abused notation and defined $\eta(x):=\eta(i)$ if $x\in \cS_i$, $i\in \cQ$.

Note that $W(x,y)$ can be defined arbitrarily on the set $\{(x,y)\in \cI^2 \ : \ \eta(x)=\eta(y)\}$, since this is a $\mu\times \mu$--null subset of $\cS^2$. Therefore $W$ is well--defined $\mu\times \mu$--a.e. We might think of the above as attaching kernels to each point of the countable set $\cQ$, with trivial kernels (consisting of a single point) attached to each point of $\cI$. The map $\eta$ should be thought of as a linear ordering of $\cQ\cup \cI$. 

Notation will be abused in a few ways. First, if $\sum_{i\in \cQ}\alpha_i=1$, then the measure $\mu$ has no support on $\cI$, so in this case we will drop $\cI$ from the notation altogether.
Second, particularly when $\sum_{i\in \cQ}\alpha_i=1$, we might drop $\eta$ from the notation and simply write $\bdirsum{\alpha_iW_i}{i\in \cQ}$. In this case, one might think of $\cQ$ as an ordered set, with the order being the one induced by $\eta$.

If $W$ and $W'$ are kernels and $r\in (0,1)$, we denote by $rW\dirsum{}{}(1-r)W'$ the direct sum with $\cQ=\{0,r\}$, $\alpha_0=r$, $\alpha_r=(1-r)$, $W_0=W$, $W_r=W'$ and $\eta : \cQ \to [0,1]$ the identity map. Note that $\alpha_0+(1-\alpha_r)=1$, so we may ignore $(\cI,\mu_{\cI})$ in this case.

\begin{remark} 
 If the probability spaces $(\cI,\mu_{\cI})$ and $(\cS_i, \mu_{i})$ are assumed to be standard, then there exists measure isomorphisms $(\cI,\mu_{\cI})\to ([0,1],\Leb)$ and $(\cS_i, \mu_{i}) \to ([0,1],\Leb)$. Rescaling using the weights $\alpha_i$ and $1-\sum_{i\in \cQ}\alpha_i$, the map $\eta$ provides a way of realising the disjoint union $\cI \sqcup \left(\bigsqcup_{i\in \cQ}\cS_i \right)$ as a partition of $[0,1]$. An example of this is shown in Figure \ref{fig:decomposedkernel}.\label{rem:embed}
\end{remark}

\begin{figure}
\centering 
\includegraphics[width=0.6\textwidth]{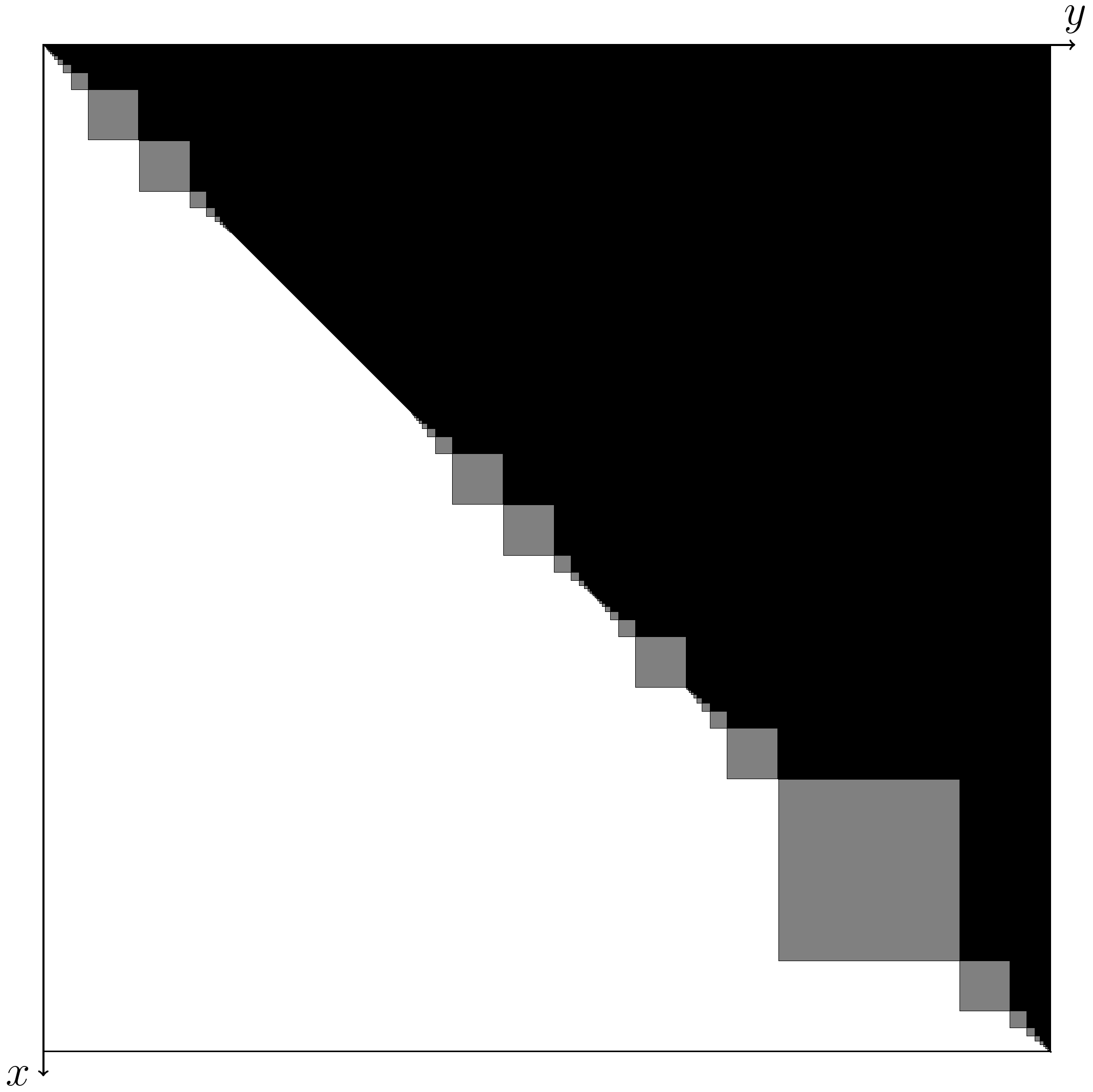}
\caption{A representation of the direct sum of a countable number of kernels. Each grey square can be seen as a kernel $[0,1]^2 \to [0,1]$, rescaled according to the weights $(\alpha_i)_{i\in \cQ}$. The kernel equals $1$ on the black region and $0$ on the white region. }
\label{fig:decomposedkernel}
\end{figure}

% 
% \begin{definition} \label{def:Wdirsum}
%  Let $\cQ\subseteq \Q$ be a countable set and let $\cI\subseteq \overline{\cQ}\setminus \cQ$. Let $\{W_i:(\cT_i,\mu_i)\to [0,1]\}_{i\in \cQ}$ be a set of tournament kernels. Let $\{\alpha_i\}_{i\in \cQ}$ be a set of positive real numbers such that $\sum_{i\in \cQ}\alpha_i\leq 1$. Provided $\mu(\cI)>0$, let $\lambda_{\ast}$ be the uniform continuous probability measure on the set $\cI$ induced by Lebesgue measure. \todo{Else?} Define a probability measure $\mu$ on $\cI\sqcup \left(\bigsqcup_{i\in\cQ}\cT_i\right)$ by $\mu=\sum_{i\in\cQ}\alpha_i\mu_i+\left(1-\sum_{i\in\cQ} \alpha_i \right)\lambda_{\ast}$.
% 
% Define the label map $L:\cS\to \cQ\cup \cI$ by
% \begin{align}
% L(x)= 
% \begin{cases}
%   i, & x\in \cT_i, i\in \cQ \\
%   x, & x\in \cI.
%  \end{cases}
% \end{align}
% 
% Define the direct sum $W= \bigoplus_{i\in\cQ\cup \cI}\alpha_iW_i :\cS^2\to [0,1]$ by
% \begin{align}
%  \bigoplus_{i\in\cQ\cup \cI}\alpha_iW_i = 
% \begin{cases} W_i(x,y) , &L(x)=L(y)\in \cQ \\
% 0, & L(x)<L(y)  \\
% 1, & L(x)>L(y).
% \end{cases}
% \end{align}
% \end{definition}
% We remark that $(\alpha_i,W_i)$ is left undefined for $i\in \cI$, but we may think of these as ``trivial'' kernels. Also, we may define $W(x,y)$ arbitrarily when $L(x)=L(y)\in \cI$, since $\{(x,y)\in \cI^2 \ : \ L(x)=L(y)\}$ is a null subset of $\cS^2$, which follows from the fact that $L_1$ is injective.

From now on, fix some arbitrary probability space $(\cS,\mu)$ and a tournament kernel $W:(\cS^2,\mu)\to [0,1]$. Our aim is to prove that $W$, $\cS$ and $\mu$ can be simultaneously decomposed into a direct sum like above. The proof follows the following steps. First we introduce an order relation $\prec$ on subsets of the set $\cS$. Using this we define a certain $\sigma$--algebra $\sigma(\cZ)$. This will define a probability space $(\cS,\sigma(\cZ),\mu)$, where the atoms  correspond to irreducible subkernels labelled by $\cQ$, and the non--atomic part to $\cI$. This will require some technical results, but after this is done we are able to show that the non--atomic part of $\sigma(\cZ)$ can be interlaced, via a map $\Lambda:\cS\to [0,1]$, in the ordering defined by the atoms of $\sigma(\cZ)$. This will give us our decomposition result. Note that the following notions typically are defined with respect to $W$, but we will suppress this dependence in our notation.

\begin{definition}
 Let $W$ be a tournament kernel on a probability space $(\cS,\mu)$. Let $A,B\subseteq \cS$ be measurable subsets with $\mu(A\cap B)=0$. We write $A\prec B$ if and only if $W(x,y)=1$ for $\mu$--a.e. $(x,y)\in A\times B$. Furthermore, we write $A\preceq B$ if and only if $A\prec B$ or $A=B$ $\mu$--a.e.
\end{definition}

We will often use without mention the fact that $\prec$ is closed under restriction to subsets, i.e. if $A'\subseteq A, B'\subseteq B$ and $A\prec B$, then $A'\prec B'$, and that $\prec$ is anti--symmetric in the sense that $A\prec B$ and $B\prec A$ if and only if $\mu(A)=0$ or $\mu(B)=0$.

\begin{definition}\label{def:irreducible}
 Let $W:\cS^2\to [0,1]$ be a tournament kernel. We say that $W$ is \emph{reducible} if there exists a measurable $B\subseteq \cS$ with $0<\mu(B)<1$ such that $B\prec \cS\setminus B$.
 Otherwise $W$ is \emph{irreducible}. 
\end{definition}

Equivalently, $W$ is reducible if and only if there exists $r\in (0,1)$ and tournament kernels $W_1,W_2$ such that $W=r W_1 \dirsum{}{}(1-r)W_2$.

Each irreducible tournament gives rise to a family of irreducible tournament kernels in the following way. Consider any irreducible tournament $G=(V(G),E(G))$, where $V(G)=\{1,2,\dots, n\}$, $n\geq 3$, and arbitrary tournament kernels $(W_i:[0,1]^2, \Leb) \to [0,1])_{i\in V(G)}$. Define the tournament kernel $W:([0,1]^2,\Leb)\to [0,1]$ by 
\begin{align}
 W(x,y) = 
\begin{cases} W_i\left(nx-i,ny-i \right) &\text{ if } (x,y)\in \bigcup_{i=1}^n \left(\frac{i-1}{n},\frac{i}{n} \right)^2 \\
 1 &\text{ if } x\in \left(\frac{i-1}{n},\frac{i}{n} \right), y \in \left(\frac{j-1}{n},\frac{j}{n} \right), i\neq j \text{ and }  ij \in E(G) \\
 0 & \text{ if } x\in \left(\frac{i-1}{n},\frac{i}{n} \right), y \in \left(\frac{j-1}{n},\frac{j}{n} \right), i\neq j \text{ and }  ij \notin E(G). \\
\end{cases}
\end{align}
This is well--defined $\mu\times \mu$--a.e., and it is not difficult to show that $W$ is an irreducible tournament kernel. Figure \ref{fig:irreduciblekernel} shows such an example where $G=\sC_3$.

\begin{figure}
\centering
\includegraphics{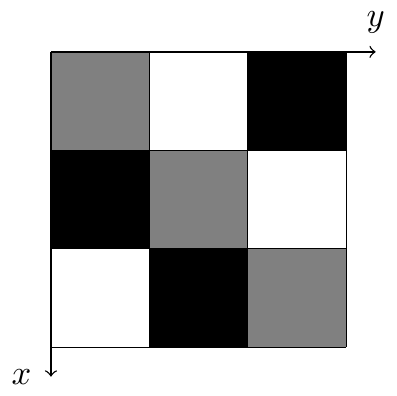}
\caption{Let $W:([0,1]^2,\mu)\to [0,1]$ be equal to $1$ on the black regions, equal to $0$ on the white regions, and equal to $1/2$ on the grey regions. This is an example of an irreducible tournament kernel. The values on the edges (constituting a null set) may be assigned arbitrarily. }
\label{fig:irreduciblekernel}
\end{figure}

\begin{definition}
 Let $\cZ$ be the set of all measurable $A\subseteq \cS$ for which there is a partition $\{B,C\}$ of $\mathcal{S}\setminus A$ with $B\prec A, A\prec C$ and $B\prec C$. When this happens we say that $(B,C)$ witnesses $A\in \cZ$. We allow $A=\emptyset, B=\emptyset$ or $C=\emptyset$. 
\end{definition}
Note that $\emptyset,\cS \in \cZ$ and that if $(B,C)$ witnesses that $A\in \cZ$, then also $B\in \cZ$ and $C\in \cZ$. The elements of $\cZ$ are sometimes called ``intervals'', see for instance \cite{CourcelleDelhomme2008}. Since $A\prec B \Rightarrow A'\prec B'$ if $A=A'$ and $B=B'$ a.e., we have that $A\in \cZ \Rightarrow A'\in \cZ$ if $A=A'$ a.e.. The next lemma shows that $\cZ$ is closed under countable intersections. 

\begin{lemma}\label{lem:semialg}
 $\cZ$ is a semi--algebra that is closed under countable intersection.
\end{lemma}

\begin{proof}
Let us first prove closure under countable intersection. Suppose $A_i\in \cZ$ for $i\geq 1$, and suppose that $(B_i,C_i)$ witnesses that $A_i\in \cZ$ for $i\geq 1$. Let $A=\bigcap_{i=1}^{\infty}A_i, B=\bigcup_{i=1}^{\infty}B_i$ and $C=\bigcup_{i=1}^{\infty}C_i$. We claim that $(B\setminus C,C)$ witnesses that $A\in \cZ$.

For any $i\geq 1$, we have that $B_i\prec A_i$, which implies $B_i\prec A$, so $B\prec A$. Hence $B\setminus C\prec A$. Similarly $A\prec C$.  Now, 
\begin{align}
B\setminus C
=\bigcup_{i=1}^{\infty}B_i \setminus \bigcup_{i=1}^{\infty}C_i 
=\bigcup_{i=1}^{\infty}B_i \cap (\cS \setminus \bigcup_{i=1}^{\infty}C_i) 
&=\bigcup_{i=1}^{\infty}B_i \cap \bigcap_{i=1}^{\infty}(\cS\setminus C_i) \\
&=\bigcup_{i=1}^{\infty}B_i \cap \bigcap_{i=1}^{\infty}(A_i\cup B_i) \\
&\subseteq \bigcap_{i=1}^{\infty}(A_i\cup B_i)\\
&\prec \cup_{i=1}^{\infty}C_i =C.
\end{align}
It remains to show that $\{B\setminus C, A, C \}$ is a partition of $\cS$ into disjoint sets. First, note that
\begin{align}
\cS \setminus A= \cS\setminus \bigcap_{i=1}^{\infty}A_i=\bigcup_{i=1}^{\infty}(\cS\setminus A_i)=\bigcup_{i=1}^{\infty}(B_i\cup C_i)=B\cup C.
\end{align}
This implies that they cover $\cS$, and moreover that $A\cap C=\emptyset$ and $A\cap B=\emptyset$. Hence also $A\cap (B\setminus C)=\emptyset$. The obvious fact $(B\setminus C)\cap C=\emptyset$ now shows the elements of the partition are pairwise disjoint. 
% 
%  Finally we show that they cover $\cS$. Suppose that $x\in \cS$ but $x\notin A$. Then $x\notin A_i$ for all $i\geq 1$, so $x\in B_i\cup C_i$ for all $i\geq 1$. If $x\in C_i$ for some $i\geq 1$, then $x\in C$. If not, then $x\in B_i$ and $x\notin C_i$ for all $i\geq 1$, so $x\in B\setminus C$. Hence $\{B\setminus C, A, C\}$ is a partition of $\cS$, so $(B\setminus C,C)$ witnesses $A\in \cZ$.

We mentioned already that $\emptyset,\cS \in \cZ$, and the above shows also that $\cZ$ is closed under intersection. To verify that $\cZ$ is a semi--algebra it remains to show that the set difference of any of its elements is a finite union of union of disjoint elements in $\cZ$. Suppose $(A_1,A_2)$ and $(B_1,B_2)$ witness that $A,B\in \cZ$ respectively. Then one can verify that $B\setminus A = (B\cap A_1)\cup (B\cap A_2)$, which is a union of pairwise disjoint elements in $\cZ$. This completes the proof.
\end{proof}

% Alternatively, it is possible to show that $\mu(B\cap C)>0$ only in the trivial case that $\mu(A)=0$, so one can show that $(B,C)$ witnesses $A\in \cZ$ unless $\mu(A)=0$ (in which case $A\in \cZ$ trivially).

The following lemma allows us to approximate arbitrary $\sigma(\cZ)$--measurable events by a finite union of pairwise disjoint elements in $\cZ$. We denote by $A\triangle B := (A\cup B)\setminus (A\cap B)$ the symmetric difference between $A$ and $B$.

 \begin{lemma}\label{lem:approx}
Given any $B\in \sigma(\cZ)$ and any $\varepsilon>0$, there exists pairwise disjoint $A_1,\dots, A_n\in \cZ$ such that $\mu(B \bigtriangleup \bigcup_{i=1}^nA_i)<\varepsilon$.
\end{lemma}

\begin{proof}
This follows from \cite[proof of Theorem 1.18]{LiebLoss2001}; it suffices that $\cZ$ is a semi--algebra, which follows from Lemma \ref{lem:semialg}.
\end{proof}

An element $A\in \sigma(\cZ)$ is said to be an \emph{atom} if there does not exist a set $A'\in \sigma(\cZ)$ with $A'\subseteq A$ and $0<\mu(A')<\mu(A)$. By a general result, the probability space $(\cS, \sigma(\cZ), \mu)$ has a decomposition into countably many atoms $S_1,\dots, S_N$ for some $N=0,1,\dots, \infty,$ and a non--atomic part $\cI$ (which we may define as the complement of the union of all atoms in $\cS$). 

We will study the atomic and non--atomic part of $(\cS,\sigma(\cZ),\mu)$, starting with the former. First, we show that any atom must be an element of $\cZ$; this is a more general result which follows from Lemma \ref{lem:semialg} and Lemma \ref{lem:approx}. Second, we show that the restriction of $W$ to any atom results in an irreducible subkernel. Third, we show that $\preceq$ induces a linear order on the atoms. After these three results we turn our attention to the non--atomic part $\cI$.

\begin{lemma}\label{lem:atom}
Let $A$ be an atom in $(\cS,\sigma(\cZ),\mu)$. Then $A\in \cZ$.
\end{lemma}

\begin{proof}
Let $A$ be an atom in $(\cS,\sigma(\cZ),\mu)$. By Lemma \ref{lem:approx}, for any $k>1$ there exist some $n=n(k)$ and mutually disjoint $A^{(k)}_1,\dots, A^{(k)}_{n}\in \cZ$ such that $\mu\left(A \triangle \bigcup_{i=1}^{\infty} A^{(k)}_i\right) < 1/k$. Since the $A^{(k)}_i$ are disjoint and $A$ is an atom, $\mu(A^{(k)}_i\cap A)=\mu(A)$ for exactly one $i$, while $\mu(A^{(k)}_i\cap A)=0$ for all other $i$. Assume without loss of generality that $\mu(A^{(k)}_1\cap A)=\mu(A)$. Then
\begin{align*}
 \mu\left(A \triangle A_1^{(k)} \right) \leq  \mu\left(A \triangle \bigcup_{i=1}^{\infty} A^{(k)}_i\right) < 1/k.
\end{align*}
By closure under countable intersections, $A':=\bigcap_{k=1}^{\infty}A_1^{(k)}$ is an element of $\cZ$. One verifies that $\mu(A\triangle A')=0$ by a limit argument. Since $A$ is a.e. equal to an element of $\cZ$, we have $A\in \cZ$, as desired.
\end{proof}

\begin{lemma}
 Let $A$ be an atom of $(\cS,\sigma(\cZ),\mu)$. Then the restriction $W|_{A}:\left(A^2, \frac{1}{\mu(A)}\mu\right)\to [0,1]$ is an irreducible tournament kernel. 
\end{lemma}

\begin{proof}
Suppose for contradiction that $W|_A$ is reducible. Then there exist a partition of $A$ into non--null disjoint $A_1,A_2\subset A$ such that $A_1\prec A_2$. But $A$ is an element of $\cZ$, so there exist $B,C$ such that $(B,C)$ witnesses that $A\in \cZ$. It follows that $(B\cup A_1,C)$ witnesses that $A_2\in \cZ$, contradicting the assumption that $A$ is an atom.
\end{proof}

\begin{lemma}\label{lem:kernelorder}
The relation $\preceq$ is a linear order on the atoms $(\cS_i)_{i=1}^N$ of $(\cS,\sigma(\cZ),\mu)$.
\end{lemma}

\begin{proof}
\textbf{Antisymmetry: } Suppose $\cS_i\prec \cS_j$ and $\cS_j\prec \cS_i$. Then $\mu(\cS_i)=0$ or $\mu(\cS_j)=0$, contradicting the fact that $\cS_i,\cS_j$ are atoms.

 \textbf{Totality: } We have $\cS_i\preceq \cS_i$ by definition. Suppose $\cS_i,\cS_j$ be distinct atoms. There exist partitions $\{A_1,A_2\}$ and $\{B_1,B_2\}$ of $\cS\setminus \cS_i$ and $\cS\setminus \cS_j$ respectively, such that $(A_1,A_2)$ witnesses $\cS_i\in \cZ$ and $(B_1,B_2)$ witnesses $\cS_j\in \cZ$. Since $\cS_i$ is an atom we have that $\mu(\cS_i\cap B_1)=0$ or $\mu(\cS_i\cap B_2)=0$. Suppose the former. Then $\mu(\cS_i\cap B_2)=\mu(\cS_i)$, i.e. $\cS_i\subseteq B_2$ a.e. But then $\cS_j\prec B_2 \Rightarrow \cS_j\prec \cS_i$. 

\textbf{Transitivity: } Let $\cS_i,\cS_j,\cS_k$ be distinct atoms and suppose $\mathcal{S}_i\prec \mathcal{S}_j$ and $\mathcal{S}_j\prec \mathcal{S}_k$. Suppose for contradiction that $\mathcal{S}_k\prec \mathcal{S}_i$. Then there exists a non--null partition $B,C\subseteq \mathcal{S} \setminus \mathcal{S}_i$ such that $\mathcal{S}_i\prec C$, $B\prec \mathcal{S}_i$, $B\prec C$. But then $\mathcal{S}_j\subseteq C$ a.e. and $\mathcal{S}_k\subseteq B$ a.e., whence $\cS_k\prec \cS_j$, a contradiction. Hence $\cS_i \prec \cS_k$.

\end{proof}

This completes our investigation of the atoms of $(\cS,\sigma(\cZ),\mu)$; we turn now our attention to the non--atomic part $\cI$. Ultimately our aim is to show that $t(\sC_3,W|_{\cI})=0$, i.e. that $W$ restricted to $\cI$ is transitive. Mainly for notational brevity, introduce a measure $\nu$ defined by
\begin{align}
 \nu(A\times B \times C):= \int_{A\times B \times C}W(x,y)W(y,z)W(z,x)d\mu(x)d\mu(y)d\mu(z) 
\end{align}
for any $A,B,C\in \cZ$. This extends to a measure on the product space $\prod_{i=1}^3 (\cS,\sigma(\cZ))$ which is absolutely continuous with respect to $\mu^3$, and $W(x,y)W(y,z)W(z,x)$ is the Radon--Nikodym derivative $d\nu / d\mu^3$. This measure is defined so that $t(\sC_3,W)=\nu(\cS^3)$, so it will suffice to show that $\nu(\cS^3)=0$.

For arbitrary measurable $A,B,C\subseteq \cS$, it holds trivially that $\nu(A\times B\times C)\leq \mu(A)\mu(B)\mu(C)$. 
If $A,B,C\in \cZ$ we can do slightly better. Namely, if $A_1,A_2,A_3\in \cZ$ are such that $\mu(A_i \cap A_j)=0$ for some $i,j=1,2,3$ with $i\neq j$, then one can show that $\nu(A_1\times A_2\times A_3)=0$. Using this fact along with a partition of $A\cup B\cup C$ into mutually disjoint elements of $\cZ$, it follows that 
\begin{align}
 \nu(A\times B\times C)\leq \mu(A\cap B\cap C)^3 \label{eq:nuineq1}
\end{align}
for any $A,B,C\in \cZ$. 

We can now prove that the restriction of $W$ to the non--atomic part $\cI$ is transitive. The idea behind the proof of this result is the following. First partition $\cI$ into many small mutually disjoint measurable subsets, each of which can be approximated like in Lemma \ref{lem:approx}. This allows us to cover $\cI$ by many small elements of $\cZ$, in such a way that their overlap is small. Using \eqref{eq:nuineq1} we see that the main contribution to $\nu(\cI^3)$ must happen ``along the diagonal'', which has small measure. Taking finer and finer partitions it follows that $\nu(\cI^3)=0$, which implies $t(\sC_3,W|_{\cI})=0$.

\begin{lemma}\label{lem:trans}
Let $\cI$ be the non--atomic part of $(\cS,\sigma(\cZ),\mu)$. Then $t(\sC_3,W|_{\cI})=0$.
\end{lemma}

\begin{proof}
For notational convenience we assume that $(\cS,\sigma(\cZ),\mu)$ is non--atomic, so that $\cS=\cI$, noting that the same proof carries through when restricting to the non--atomic subspace $\cI$, if necessary. Since $(\cS,\sigma(\cZ),\mu)$ is non--atomic, there exists pairwise disjoint measurable $B_1,\dots, B_n\in \sigma(\cZ)$ such  $\mu(B_i)=1/n$ for every $i=1,\dots, n$. 

By Lemma \ref{lem:approx}, for each $i=1,\dots, n$, there exists some $m_i\in \N$ and a set $A_i:=\bigcup_{\ell=1}^{m_i}A_{\ell}^{(i)}$, where $\left(A_{\ell}^{(i)}\right)_{i=1}^{m_i}\subseteq \cZ$ are pairwise disjoint,  satisfying $\mu(A_i \bigtriangleup B_i)<1/n^4$. 
Let $m=\max\{m_i \ : \ i=1,\dots n\}$, and for any $m_i< \ell \leq m$, define $A_{\ell}^{(i)}=\emptyset$, so that we may write $A_i=\bigcup_{\ell=1}^mA_{\ell}^{(i)}$ for each $i=1,\dots, n$.

Using the above, \eqref{eq:nuineq1}, the union bound and various trivial bounds on lower order cross terms (but still leaving out many details and not worrying about the tightness of the bound), we obtain
\begin{align}
 \nu(\cS^3) = \nu\left(\left(\bigcup_{i=1}^n B_i \right)^3\right) 
&\leq \nu\left(\left(\bigcup_{i=1}^n A_i \right)^3\right) + o_n(1) \\
&=\nu\left(\left(\bigcup_{i=1}^n\bigcup_{\ell=1}^m A_{\ell}^{(i)} \right)^3\right)  +o_n(1) \\
&\leq \sum_{i=1}^n \sum_{\ell_1,\ell_2,\ell_3=1}^m\nu(A_{\ell_1}^{(i)}\times A_{\ell_2}^{(i)}\times A_{\ell_3}^{(i)})+o_n(1) \\
&= \sum_{i=1}^n \sum_{\ell=1}^m \mu\left(A_{\ell}^{(i)}\right)^3+o_n(1) \\
&\leq \sum_{i=1}^n \left(\sum_{\ell=1}^m\mu\left(A_{\ell}^{(i)}\right)\right)^3+o_n(1) \\
&\leq \sum_{i=1}^n \left(\frac{1}{n}+\frac{1}{n^4}\right)^3 +o_n(1) \\
&=o_n(1),
\end{align}
where $f=o_n(1)$ means that $f(n)\to 0$ as $n\to \infty$. Since $n$ is arbitrary, it must be that $t(\sC_3,W)=\nu(\cS^3)=0$.
\end{proof}

Since transitivity is a hereditary property, it follows that $W$ restricted to any non--null subset of $\cI$ is also transitive. 

We now turn to presenting a canonical embedding of the the atoms $\cS_i$ and the non--atomic part $\cI$ inside $[0,1]$. Associate to every $x\in \cS$ a set 
\begin{align}
 R_x=\{j \in \cQ \ : \ W(x,y)=1 \text{ for } \mu\text{--a.e. } y \in \cS_j  \},
\end{align}
noting that, for almost every $x\in \cI$ and all $i\in \cQ$, either $W(x,y)=1$ for almost every $y\in \cS_i$ or $W(x,y)=0$ for almost every $y\in \cS_i$. Also note that for any $i\in \cQ$ and a.e. $x\in \cS_i$ we have $R_x=\{j\in \cQ \ : \ \cS_i \prec \cS_j\}$. 

Define $\Lambda : \cS\to [0,1]$ by
\begin{align}
 \Lambda(x) = \int_{\cI}W(x,y)d\mu(y) + \sum_{j\in R_x}\alpha_j.
\end{align}
The function $\Lambda$ is a.e. constant on each $\cS_i$, and for a.e. $(x,y)\in \cS_i\times \cS_j$, $i\neq j$, we have $\Lambda(x)>\Lambda(y)$ if and only if $\cS_i\prec \cS_j$. Define $\eta:\cQ\cup \cI \to [0,1]$ by
\begin{align}
 \eta(i) = 
\begin{cases}
1-\Lambda(i) & i\in \cI, \\
1-\frac{1}{\mu(\cS_{i})}{\int_{\cS_{i}}}\Lambda(x)d\mu(x)-\mu(\cS_i) & i\in \cQ.
\end{cases}
\end{align}
The map $\eta$ has the following properties.
\begin{itemize}
 \item $\eta|_{\cQ}$ is injective.
 \item The intervals $[\eta(i),\eta(i)+\mu(\cS_i))$, $i\in \cQ$ are disjoint.
 \item $\eta(i)<\eta(j)$ if and only if $\cS_i\prec \cS_j$.
 \item $\eta|_{\cI}$ is injective up to $\mu_{\cI}$--null sets.
 \item $\eta|_{\cI}$ is supported up to $\mu_{\cI}$--null sets on $[0,1]\setminus \bigcup_{i\in \cQ}[\eta(i),\eta(i)+\mu(\cS_i))$
 \item The sets $\{[\eta(i),\eta(i)+\mu(\cS_i))\}_{i\in \cQ} \cup \eta(\cI)$ form a partition (up to $\mu_{\cI}$--null sets of $[0,1]$). 
\end{itemize}
Note that the last property implies several of the other. The map $\eta$ provides a canonical embedding of $\cQ\cup \cI$ into $[0,1]$. Note that $\eta|_{\cI}:(\cI,\mu)\to (\eta(\cI),\Leb)$ is in fact a measure--preserving bijection, so taking any measure--preserving bijections $(\cS_i,\mu)\rightarrow ([\eta(i),\eta(i)+\mu(\cS_i),\Leb)$ for $i\in \cQ$ would give us a measure--preserving bijection (up to null sets) of $(\cS,\mu)$ into $([0,1],\Leb)$. 

The following theorem summarises the above results.

\begin{theorem}\label{thm:Wdecomp}
 Let $W:(\cS^2,\mu)\to [0,1]$ be a tournament kernel. Then $\cS$ has a partition $\cI \sqcup \bigsqcup_{i\in\cQ}\cS_i$ (where either $\cI$ or $\cQ$ may be empty), the measure $\mu$ can be written as 
\begin{align}
 \mu = \left(1-\sum_{i\in \cQ}\mu(\cT_i)\right)\mu_{\cI} + \sum_{i\in \cQ}\mu(\cT_i)\mu_i 
\end{align}
and $W$ as 
\begin{align}
 W=\left(\bdirsum{\mu(\cS_i)W_i}{i\in \cQ},\cI,\eta \right),
\end{align}
where $\eta: \cQ \cup \cI \to [0,1]$ is defined like above. Furthermore, the following statements hold (with the obvious modificiations if $\cQ$ or $\cI$ is empty).
\begin{enumerate}[(i)]
 \item $\mu(\cS_i)>0$ for each $i\in \cQ$. 
 \item The measures $\mu_{i}$ are the probability measures induced by $\mu$ on $\cS_i$, i.e. $\mu_i(A)=\frac{\mu(A \cap \cS_i)}{\mu(\cS_i)}$ for each measurable $A\subseteq \cS$. 
 \item The measure $\mu_{\cI}$ is the probability measure induced by $\mu$ on $\cI$ and it is non--atomic.
 \item The restriction $\eta|_{\cQ}:\cQ \to [0,1]$ is injective and the restriction $\eta|_{\cI}$ is injective (up to null sets) and supported on $[0,1]\setminus \bigcup_{i\in \cQ} [\eta(i),\eta(i)+\alpha_i)$.
 \item For $i\in \cQ$, the tournament kernel $W_i:(\cS_i,\mu_i)\to [0,1]$ is irreducible and equal to $W|_{\cS_i}:(\cS_i,\frac{1}{\mu(\cS_i)}\mu)\to [0,1]$. 
 \item For any measurable $A\subseteq \cI$, the restriction $W|_{A}:(A,\frac{1}{\mu_{\cI}(A)}\mu_{\cI})\to [0,1]$ is transitive.
\end{enumerate}
\end{theorem}

We should think of this theorem as a \emph{reordering} via the map $\eta$ of the original set $\cS$. In this sense Theorem \ref{thm:Wdecomp} gives the existence of a relabelling of the points of $\cS$, such that $W$ is equal to a direct sum in the ``natural way''. Theorem \ref{thm:tourdecomp1} says something similar. Given a \emph{labeled} graph, there is a relabeling of the vertices giving rise to the decomposition in Theorem \ref{thm:tourdecomp1} in a ``natural way''. Indeed, the proof technique used to prove Theorem \ref{thm:Wdecomp} could also have been used to prove Theorem \ref{thm:tourdecomp1}, with the slight difference that it suffices to consider just a $\sigma$--algebra and its atoms (no measures need to be involved). For $\sigma$--algebras an \emph{atom} would be defined as a set in the $\sigma$--algebra which does not contain any proper subset also in the $\sigma$--algebra. In that case the atoms would correspond to the irreducible subtournaments. 

In Theorem \ref{thm:tourdecomp1} and Corollary \ref{cor:tourdecomp} we saw that the irreducible components of size $1$ could sometimes be merged to form larger transitive components. One might ask to what extent this is possible here as well. Is it the case that the non--atomic part $\cI$ can be written as a countable union of non--null sets in $\cZ$, which can then be interlaced into the ordering $(\cQ,\eta)$ as transitive kernels? If this question is answered in the affirmative, then one could afford to forget about $\cI$ in Theorem \ref{thm:Wdecomp}. For instance, the following two direct sums are equivalent.
\begin{enumerate}
 \item Let $\cQ=\{0,3/4\}$, $\alpha_0=1/4$, $\alpha_{3/4}=1/4$, $\cI=(1/4,1/2)$, $\mu_{\cI}=4\Leb$ and let $W_0,W_{3/4}$ be arbitrary tournament kernels. Let $\eta$ be the identity map.
 \item Let $\cQ=\{0,1/4,3/4\}$, $\alpha_0=1/4$, $\alpha_{1/4}=1/2$, $\alpha_{3/4}=1/4$, and let $W_0,W_{3/4}$ be as above, with $W_{1/4}$ being the transitive kernel. Let $\eta$ be the identity map.
\end{enumerate}
In the second case $\alpha_0+\alpha_{1/4}+\alpha_{3/4}=1$, so we did not need to define $(\cI,\mu_{\cI})$. This example makes it clear that introducing $(\cI,\mu_{\cI})$ is sometimes ``unnecessary''. However, the following example of a Cantor--type tournament kernel shows that there is a direct sum for which $\sum_{i\in \cQ}\alpha_i<1$, but $\cI$ is not a countable union of elements from $\cZ$. In fact, the set $\cI$ contains no non--null element of $\cZ$ as a subset. This demonstrates that, in general, one cannot hope to forget about $(\cI,\mu_{\cI})$ by decomposing it into a countable number of parts which can be interlaced in the ordered set $\eta(\cQ)$. 

\begin{example}\label{ex:cantor}
Let $C_0=[0,1]$. Proceeding inductively, construct $C_{n+1}$ by removing from $C_n$ subintervals of width $1/2^{2(n+1)}$ from the middle of each of its $2^{n}$ intervals. Let $\cI=\bigcap_{n=0}^{\infty}C_n$. This is an uncountable and nowhere dense set with measure $1/2$ known as a Smith--Volterra--Cantor set. Its complement $[0,1]\setminus \cI$ is a countably infinite collection of disjoint intervals. Let $\cQ$ be the set of midpoints of these intervals, and for $i\in \cQ$, let $a_i<b_i$ be the endpoints of the interval containing $i$. Let $\alpha_i=b_i-a_i$ be the width of the interval. Then $\cQ$ is a countable subset of $[0,1]$ and $\cI\subseteq \overline{\cQ}\setminus \cQ$. 

Let $W':([0,1]^2,\Leb)\to [0,1]$ be some arbitrary irreducible tournament kernel. For each $i\in \cQ$, let $W_i:([a_i,b_i]^2,\mu_i)\to [0,1]$, where $\mu_i$ is the probability measure on $[a_i,b_i]$ induced by Lebesgue measure, be the kernel defined by $W_i(x,y)=W'\left(\frac{x-a_i}{b_i-a_i},\frac{y-a_i}{b_i-a_i} \right)$. Furthermore, let $\mu_{\cI}$ be the measure on $\cI$ defined by $\mu_{\cI}(\cdot)=\Leb(\cI\cap \cdot)/\Leb(\cI)$. Then let $\mu=\sum_{i\in \cQ}\alpha_i\mu_i + \frac{1}{2}\mu_{\cI}$ and define $W:([0,1]^2,\mu)\to [0,1]$ as the direct sum
\begin{align}
\left(\bdirsum{\alpha_iW_i}{i\in \cQ},\cI,\eta \right),
\end{align}
where $\eta:\cQ \cup \cI \to [0,1]$ is defined by 
\begin{align}
 \eta(i)=\begin{cases} a_i, & i\in \cQ \\ i, & i\in \cI.\end{cases}
\end{align}
This example can be generalized to construct direct sums of kernels for which $\sum_{i\in \cQ}\alpha_i$ takes any value in $(0,1)$. 
\end{example}
One could still refine Theorem \ref{thm:Wdecomp} by interlacing in the ordering $(\cQ,\eta)$ the maximal subsets $B\subseteq \cI$ such that $B\in \cZ$. The restriction of $W$ to any such $B$ is, as mentioned, a transitive tournament kernel. However, we refrain from doing this, since it does not bring any technical advantages. Instead we turn to providing formulae for induced densities of direct sums of tournament kernels, analogous to Theorem \ref{lem:tourdens}. We do this in two steps; first we prove a formula under the assumption that $\sum_{i\in \cQ}\alpha_i=1$ (so there is no non--atomic part); second, we extend this to the case $\sum_{i\in \cQ}\alpha_i<1$.

Given a finite tournament $F$ and a countable (ordered) set $\cQ$, denote by $\cP(F,Q)$ the set of all decompositions of $F$ into direct sums $\bdirsum{F_i}{i\in \cQ}$. Note that at most finitely many of the $F_i$ can be non--empty.

\begin{theorem}\label{thm:densvee1}
 Let $W=\left(\bdirsum{\alpha_iW_i}{i\in \cQ},\eta \right)$ be a direct sum of tournament kernels such that $\sum_{i\in \cQ}\alpha_i=1$. For any finite tournament $F$,
\begin{align}
 t_{\text{ind}}\left(F,W\right)=\sum_{\cP(F,\cQ)} \prod_{i\in\cQ} \alpha_{i}^{v(F_i)}t_{\text{ind}}(F_i,W_{i}).
\end{align}
\end{theorem}

\begin{proof}
Recall that  
\begin{align}t_{\text{ind}}(F,W)
&= \E \left[\prod_{(v,w)\in E(F)} W(X_v,X_w) \right].
\end{align}
where the expectation is taken with respect to the measure $\mu=\sum_{i\in \cQ}\alpha_i\mu_i$. That is, with probability $\alpha_i$, the random variable $X_v$ is chosen from $\cT_i$ according to the measure $\mu_i$, independently for all $v\in V(F)$. Let $f(X_v)$ denote the random variable given by $f(X_v)=i$ if $X_v$ is chosen from $\cT_i$. 

There is a contribution to the expectation if and only if the sets $\{v\in V(F) \ : \ f(X_v)=i\}$ induce a decomposition $F=\bdirsum{F_i}{i\in \cQ}$ with $V(F_i)=\{v\in V(F) \ : \ f(X_v)=i\}$. To compute the expectation, we therefore first fix a decomposition $F=\bdirsum{F_i}{i\in \cQ}$ and count the contribution to the expectation with respect to this fixed decomposition. Summing over all such decompositions (and implicitly using the law of total expectation) then gives the desired result. 

So fix a decomposition $F=\bdirsum{F_i}{i\in \cQ}$. The probability that $f(X_v)=i$ for all $v\in V(F_i)$ is, by independence, equal to $\prod_{i\in \cQ}\alpha_i^{v(F_i)}$. Conditional on this event, again due to independence, the contribution to the expectation is 
\begin{align}
  \prod_{i\in \cQ}\E \left[\prod_{(v,w)\in E(F_i)} W(X_v,X_w) \right] = \prod_{i\in \cQ}t_{\text{ind}}(F_i,W_i).
\end{align}
Hence the contribution to the expectation with respect to the decomposition $F=\bdirsum{F_i}{i\in \cQ}$ is
\begin{align}
 \prod_{i\in \cQ} \alpha^{v(F_i)}_i t_{\text{ind}}(F_i,W_i).
\end{align}
Summing over all possible decompositions finishes the proof.
\end{proof}

In the next theorem we extend this result to direct sums with $\sum_{i\in \cQ}\alpha_i<1$. First we introduce some notation. For any $a,b\in \cQ \cup \cI$, define 

\begin{itemize}
 \item $\cQ_{(a,b)}=\{j\in\cQ  \ : \eta(a)<\eta(j)<\eta(b)\}$,
 \item $\alpha_{(a,b)}=\sum_{j\in \cQ_{(a,b)}}\alpha_j$, and
 \item $W_{(a,b)} = \left(\bdirsum{(\alpha_i/\alpha_{(a,b)}) W_i}{i\in \cQ_{(a,b)}}, \eta|_{Q_{(a,b)}}\right)$.
\end{itemize}
 Define finally $\Delta \cI^p = \{(r_1,r_2,\dots, r_p) \ : \ \eta(r_1)<\eta(r_2)<\dots < \eta(r_p) \}$.

For any finite tournament $F$ and any $p\geq 1$, let $\cP(F,p)$ denote the set of decompositions $F=F_{0}\dirsum{}{} K(v_1) \dirsum{}{} F_{1} \dirsum{}{} K(v_2) \dirsum{}{} \dots \dirsum{}{} K(v_p)\dirsum{}{} F_{p}$ where each $F_{i}$ is a tournament (possibly empty or possibly further decomposable)) and $K(v_i)$ denotes the singleton with vertex $v_i$ and empty edge set. This sort of decomposition can be obtained by choosing $p$ vertices among the singletons in the decomposition in Theorem \ref{thm:tourdecomp1}, followed by lumping together the remaining parts appropriately.

\begin{theorem}\label{thm:densvee2}
Let $W=\left(\bdirsum{\alpha_iW_i}{i\in \cQ},\cI,\eta \right)$ be a direct sum of tournament kernels such that $\sum_{i\in \cQ}\alpha_i\leq 1$. For any finite tournament $F$, 
\begin{align}\label{eq:vee1}
t_{\text{ind}}(F,W)&=\alpha^{v(F)}t_{ind}\left(F,\bdirsum{\frac{\alpha_i}{\alpha}W_i}{i\in \cQ} \right) \\
&\qquad + \sum_{p=1}^{v(F)}\sum_{F_p}\int_{\Delta\cI^p} \binom{v(F)}{p}(1-\alpha)^p\prod_{i=0}^{p} \alpha_{(r_i,r_{i+1})}^{v(F_{i})} t_{\text{ind}}\left(F_{i}, W_{(r_i,r_{i+1})}\right) \prod_{i=1}^p \mu_{\cI}(dr_i).
\end{align}
\end{theorem}

For notational convenience, we have implicitly defined $\cQ_{(r_0,r_1)}=\{j\in \cQ \ : \ \eta(j)<\eta(r_1)\}$ and $\cQ_{(r_p,r_{p+1})}=\{j\in \cQ \ : \ \eta(r_{p})<\eta(j)\}$ to deal with the boundary cases in the product.

\begin{proof}
 We omit this proof. The idea is very similar to the proofs of Theorem \ref{thm:densvee1} and Theorem \ref{thm:glimit2}.
\end{proof}
 

%% file: master_densityformula3.tex
%%%%%%%%%%%%%%%%%%%%%%%%%%%%%%%%%%%%%%%%%%%%%%%%%%%%%%%%%%%
\section{Direct sums of tournament limits}
%%%%%%%%%%%%%%%%%%%%%%%%%%%%%%%%%%%%%%%%%%%%%%%%%%%%%%%%%%%
Throughout this section, suppose we have the following objects:
\begin{itemize}
 \item a countable set $\cQ$,
 \item a non--atomic probability space $(\cI,\mu_{\cI})$,
 \item non--negative constants $(\alpha_i)_{i\in \cQ}$ such that $\sum_{i\in \cQ}\alpha_i\leq 1$, 
 \item tournament limits $(\Gamma_i)_{i\in \cQ}$, and 
 \item an map $\eta:\cQ \cup \cI \to [0,1]$ such that $\eta|{\cQ}$ is injective and $\eta|_{\cI}$ is injective up to $\mu_{\cI}$--null sets.
\end{itemize}
The reader will note that this list of objects is identical to the list of objects required to define the direct sum of tournament kernels in Section \ref{sec:kerneldecomp}, with the exception that the tournament kernels $(W_i)_{i\in \cQ}$ have been replaced by tournament limits $(\Gamma_i)_{i\in \cQ}$.

The aim of this section is two--fold. First, we wish to define the direct sum $\left(\bdirsum{\alpha_i\Gamma_i}{i\in \cQ},\cI,\eta \right)$ directly, without direct reference to any particular choice of tournament kernels representing $(\Gamma_i)_{i\in \cQ}$.  Second, we would like that the notions of direct sums for kernels and limits agree, i.e. to show that the direct sum $\left(\bdirsum{\alpha_iW_i}{i\in \cQ},\cI,\eta \right)$ represents $\left(\bdirsum{\alpha_i\Gamma_i}{i\in \cQ},\cI,\eta \right)$, whenever $W_i$ represents $\Gamma_i$ for each $i\in \cQ$. Naturally, our definition of the direct sum of tournament limits is made so that we can achieve the second aim, so in a sense we are merely choosing the ``correct'' definition in order to make the theory sensible.

We first define the direct sum of tournament limits in the case that $\sum_{i\in \cQ}\alpha_i=1$. The main idea of the following proof is to use the set $((\alpha_i,\Gamma_i))_{i\in \cQ}$ and the injective map $\eta:\cQ\to [0,1]$ to construct a suitably chosen sequence $(G_m)_{m\in \N}$ of direct sums of tournaments, such that the numbers $t(F,G_m)$ converge for any finite $F$. We then define the direct sum to be the unique limit of $G_m$.

\begin{theorem}\label{thm:glimit1}
 Let $(\Gamma_i)_{i\in \cQ}\subseteq \widehat{\cT}$ be a set of tournament limits and let $(\alpha_i)_{i\in \cQ}$ be a set of positive numbers such that $\sum_{i\in \cQ}\alpha_i=1$. Then there is a unique tournament limit $\Gamma$ such that, for any finite tournament $F$,
 \begin{align}
 t_{\text{ind}}(F,\Gamma)=\sum_{\cP(F,\cQ)}\prod_{i\in \cQ}\alpha_i ^{v(F_i)}t_{\text{ind}}(F_i,\Gamma_{i}).
\end{align} 
Define 
\begin{align}
 \left(\bdirsum{\alpha_i\Gamma_i}{i\in \cQ},\eta \right):= \Gamma.
\end{align}
\end{theorem}

\begin{proof}
The proof follows Theorem 4.2 of Janson \cite{Janson2008}, which gives the corresponding result for graph limits. Below we implicitly give $\cQ$ the ordering induced by $\eta$.

\textbf{$\cQ$ finite. }First let $\cQ$ be finite. For each $i\in \cQ$, let $(G_{i,m})_{m\in \N}$ be a sequence of tournaments such that
\begin{align}
 G_{i,m}&\to \Gamma_i, & i\in \cQ \\
 \frac{v(G_{i,m})}{\sum_{j\in \cQ}v(G_{j,m})}&\to \alpha_i, & i\in \cQ
\end{align}
as $m\to \infty$. The graph sequence $G_{i,m}$ may be chosen deterministically or by taking initial segments of the random graph $G(\infty,\Gamma)$, which will give us almost sure convergence to $\Gamma$. We define $G_m=\bdirsum{G_{i,m}}{i\in \cQ}$.

Recall that $(n)_k\sim n^k$ if $k$ is fixed and $n\to \infty$ and that $\sum_{i\in \cQ} v(F_i)=v(F)$. By Lemma \ref{lem:tourdens} we have
\begin{align}
 t_{\text{ind}}\left(F,G_m\right)
&=\frac{1}{(v(G_m))_{v(F)}}\sum_{\cP(F,\cQ)}\prod_{i\in \cQ}v((G_{i,m}))_{v(F_i)}t_{\text{ind}}(F_i,G_{i,m}) \\
&\sim \frac{1}{v(G_m)^{v(F)}}\sum_{\cP(F,\cQ)}\prod_{i\in \cQ}v(G_{i,m})^{v(F_i)}t_{\text{ind}}(F_i,G_{i,m}) \\
&=\sum_{\cP(F,\cQ)}\prod_{i\in \cQ}\left(\frac{v(G_{i,m})}{v(G_m)}\right)^{v(F_i)}t_{\text{ind}}(F_i,G_{i,m}) \\
&\to \sum_{\cP(F,\cQ)}\prod_{i\in \cQ}\alpha_{i} ^{v(F_i)}t_{\text{ind}}(F_i,\Gamma_{i})
\end{align}
as $m\to \infty$. Note that the sum is finite (since $\cQ$ is finite), so we may interchange limit and summation as in the last step. Therefore the numbers $ t_{\text{ind}}\left(F,G_m\right)$ converge for any finite tournament $F$, so $(G_m)_{m\in \N}$ converges to some tournament limit $\Gamma$. We denote this tournament limit by $\left(\bdirsum{\alpha_i\Gamma_i}{i\in \cQ},\eta \right)$. This proves the statement for $\cQ$ finite.

\textbf{$\cQ$ countably infinite. }Now suppose that $\cQ$ be countably infinite. Fix a bijection $\Psi: \N \to \cQ$, and for $m\in \N$, define $\cQ_m:=\Psi\left(\{1,2,\dots, m\} \right)\subseteq \cQ$ and $\Gamma_{(m)} := \left(\dirsum{(\alpha_i/\alpha_{(m)})\Gamma_i}{i\in \cQ_m},\eta|_{\cQ_m} \right)$, where $\alpha_{(m)}=\sum_{i\in \cQ_m}\alpha_i$. The tournament limits $\Gamma_{(m)}$ are well--defined by the first part of the proof. We claim that the sequence $(\Gamma_{(m)})_{m=1}^{\infty}$ is convergent in the closed space $\widehat{\cT}$. In fact, by monotone convergence and the fact that $\alpha_{(m)}\to 1$ as $m\to \infty$, it holds that, for any finite tournament $F$,
\begin{align}
 t_{\text{ind}}(F,\Gamma_{(m)})
&=\sum_{\cP(F,\cQ_m)}\prod_{i\in \cQ_m}\left(\frac{\alpha_i}{\alpha_{(m)}}\right) ^{v(F_i)}t_{\text{ind}}(F_i,\Gamma_{i}) \\
&\to \sum_{\cP(F,\cQ)}\prod_{i\in \cQ}\alpha_{i} ^{v(F_i)}t_{\text{ind}}(F_i,\Gamma_{i})
\end{align}
as $m\to \infty$. This proves that the sequence $(\Gamma_{(m)})_{m\geq 1}\subseteq \widehat{\cT}$ is convergent in the closed space $\overline{\cT}$, so it converges to some limit object $\Gamma \in \overline{\cT}$. It follows from the description in Section 2 that the space $\widehat{\cT}$ is closed in $\overline{\cT}$, so this implies that $\Gamma \in \widehat{\cT}$, i.e. $\Gamma$ is a tournament limit. Moreover, it is unique since tournament limits are uniquely defined by their homomorphism densities. Define
\begin{align}
 \left(\bdirsum{\alpha_i\Gamma_i}{i\in \cQ},\eta \right) := \lim_{m\to \infty}\Gamma_{(m)}=\Gamma
\end{align}
\end{proof}

We now extend this definition to the case when $\sum_{i\in \cQ}\alpha_i<1$. We recycle the notation used in Theorem \ref{thm:densvee1} and do not restate it here. However, we should mention that the tournament limit $\Gamma_{(a,b)} := \left(\bdirsum{(\alpha_i/\alpha_{(a,b)}) \Gamma_i}{i\in \cQ_{(a,b)}}, \eta|_{Q_{(a,b)}}\right)$ is well--defined by Theorem \ref{thm:glimit1}.

%%%%%%%%%%%%%%%%%%%%%%%%%%%%%%%%%%%%%%%%%%%%%%%%%%%%%%%%%%%%

\begin{theorem}\label{thm:glimit2}
Let $(\Gamma_i)_{i\in \cQ}\subseteq \widehat{\cT}$ be a set of tournament limits and let $(\alpha_i)_{i\in \cQ}$ be a set of positive numbers such that $0<\sum_{i\in \cQ}\alpha_i\leq 1$. Then there is a unique tournament limit $\Gamma$ such that, for any finite tournament $F$,
\begin{align}
t_{\text{ind}}(F,\Gamma)=&\alpha^{v(F)}t_{\text{ind}}\left(F,\bdirsum{\frac{\alpha_i}{\alpha} \Gamma_i}{i\in \cQ} \right) \\
&\quad +\sum_{p=1}^{v(F)}\sum_{\cP(F,p)}\int_{\Delta \cI^p}\binom{v(F)}{p}(1-\alpha)^p \prod_{i=0}^{p} \alpha_{(r_i,r_{i+1})}^{v(F_{i})} t_{\text{ind}}\left(F_{i}, \Gamma_{(r_i,r_{i+1})}\right) \prod_{i=1}^p \mu_{\cI}(dr_i). \label{eq:glimit2}
\end{align}
Define  
\begin{align}
 \left(\bdirsum{\alpha_i\Gamma_i}{i\in \cQ},\cI,\eta \right):= \Gamma.
\end{align}
\end{theorem}

Essentially the idea of the proof is similar to that of Theorem \ref{thm:glimit1}, with the crucial difference that we now construct a \emph{random} sequence $(G_m)_{m=1}^{\infty}$ of tournaments. We first show that the numbers $\E[t_{ind}(F,G_m)]$ converge as $m\to \infty$ to the right hand side of \eqref{eq:glimit2}, for each tournament $F$. According to Theorem 3.1 of \cite{DiaconisJanson} (extended to digraph and tournament limits), this is equivalent to the convergence of $G_m$ in distribution to a \emph{random} tournament limit. However, by Theorem \ref{thm:densvee2}, there exists a tournament kernel $W$ (which is a direct sum) such that $\E[t_{ind}(F,G_m)]\to t_{ind}(F,W)$. Now let $\Gamma$ be the unique tournament limit representing $W$, so that $t_{ind}(F,\Gamma)=t_{ind}(F,W)$ and $\E[t_{ind}(F,G_m)]\to t(F,\Gamma)$ for any tournament $F$. This implies, by Theorem 3.1 or Corollary 3.2 of \cite{DiaconisJanson}, that $G_m$ converges to the \emph{non--random} tourmanent limit $\Gamma$ in distribution (and hence in probability). By uniqueness of limits, we may define the direct sum $\left(\bdirsum{\alpha_i\Gamma_i}{i\in \cQ},\cI,\eta \right)$ to be this $\Gamma$.

\begin{proof}
Give $\cQ$ and $\cI$ the ordering induced by $\eta$.

For each $i\in \cQ$, let $(G_{i,m})_{m\in \N}$ be a sequence of tournaments such that $G_{i,m}\to \Gamma_i$ as $m\to \infty$. Let $(G_{\ast,m})_{m\in \N}$ be a sequence of sets. These sequences can be constructed so as to satisfy
\begin{align}
 \frac{v(G_{i,m})}{\sum_{j\in \cQ}v(G_{j,m})+|G_{\ast,m}|}&\to \alpha_i, & i\in \cQ \\
\frac{|G_{\ast,m}|}{\sum_{j\in \cQ}v(G_{j,m})+|G_{\ast,m}|}&\to 1-\sum_{i\in \cQ}\alpha_i=1-\alpha
\end{align}
as $m\to \infty$. We may take deterministic tournament sequences, or sequences which converge almost surely to the correct tournament limits. Also we may take $v(G_{i,0})=0$ for all $i\in \cQ$, so that the number $\sum_{i\in \cQ}v(G_{i,m})+|G_{\ast,m}|$ is finite for any $m\in \N$.

Define $f:G_{\ast,m}\sqcup \bigsqcup_{i\in \cQ} V(G_{i,m}) \to \cQ \cup \cI$ as follows. For each $i\in \cQ$ and $v\in V(G_{i,m})$, let $f(v)=i$, while for each $v\in V(G_{\ast,m})$, let $f(v)$ be a $\mu_{\cI}$--uniform random point in $\cI$. (It is because we choose the labels $\mu_{\cI}$--uniformly that we only obtain convergence of the expected value.) Let $G_m$ be the random tournament with vertex set 
\begin{align}
 V(G_m)=G_{\ast,m}\sqcup \bigsqcup_{i\in \cQ} V(G_{i,m})
\end{align}
and edge set
\begin{align}
 E(G_m)&=\{(v,w)\in V(G_m)^2\ : \ f(v)<f(w), \text{ or } f(v)=f(w) \text{ and } (v,w)\in E(G_{f(v)})\}.
\end{align}
Note that $f(v)=f(w)$ occurs with probability $0$ if $f(v),f(w)\in \cI$ (since the probability space $(\cI,\mu_{\cI})$ is non--atomic). Also note that the sets $V(G_m)$ may be taken to be non--random, so the randomness lies in the edge sets $E(G_m)$.

As mentioned, we are interested in determining the limit of $\E[t_{ind}(F,G_m)]$ as $m\to \infty$ for any finite tournament $F$. By the discussion in Section \ref{sec:prelim}, it holds that
\begin{align}
 \E[t_{ind}(F,G_m)]&=\E[\pP[\Phi:V(F)\to V(G_m) \text{ is a homomorphism that preserves non--adjacency}|G_m]] \\
& = \pP[\Phi:V(F)\to V(G_m) \text{ is a homomorphism that preserves non--adjacency}]
\end{align}
where $\Phi:V(F)\to V(G_m)$ is chosen at random among all injective maps $V(F)\to V(G_m)$. It should be mentioned that there are two layers of randomness here -- firstly $G_m$ is random, and secondly we choose, conditional on the $G_m$, a random injective map $V(F)\to V(G_m)$. We shall compute this expectation by repeated use of conditioning and the law of total expectation.

For each $v\in V(F)$, either $\Phi(v)\in \sqcup_{i\in \cQ}V(G_{i,m})$ (equivalently $(f\circ \Phi)(v)\in \cQ$) or $\Phi(v)\in G_{\ast,m}$ (equivalently $(f\circ \Phi)(v)\in \cI$). Fix some $p\in \{0,1,\dots, v(F)\}$ and condition on the event that
\begin{align}
 |\Phi^{-1}(f^{-1}(\cI))|=p \label{eq:Ap}
\end{align}
The probability of this event can be shown to tend to
\begin{align}
 \binom{v(F)}{p}(1-\alpha)^p \alpha^{v(F)-p}
\end{align}
as $m\to \infty$. The subgraph of $G_m$ induced by the vertices $\Phi(\Phi^{-1}(f^{-1}(\cI)))$ is transitive, so $\Phi$ preserves adjacency and non--adjacency only if the vertices $\Phi^{-1}(f^{-1}(\cI))$ induce a transitive subgraph of $F$. That is, if $\Phi^{-1}(f^{-1}(\cI))=\{v_1,v_2,\dots, v_p\}$, where the indices are chosen so that $(v_i,v_j)\in E(F)$ for any $1\leq i < j \leq p$, then $\Phi$ can only preserves adjacency and non--adjacency if also
\begin{align}
 (f\circ \Phi)(v_1) < (f\circ \Phi)(v_2) < \cdots < (f\circ \Phi)(v_p),  \label{eq:Iverts}
\end{align}

Since the set $\Phi^{-1}(f^{-1}(\cI))$ must induce a transitive subgraph of $F$, the vertices $v_1,\dots, v_p\in V(F)$ must appear as singletons in the decomposition given in Theorem \ref{thm:tourdecomp1}. That is, the vertices $v_1,\dots v_p$ must induce a decomposition of the form $F=F_0\dirsum{}{} K(v_1) \dirsum{}{} F_1 \dirsum{}{} \cdots \dirsum{}{} K(v_p) \dirsum{}{} F_{p}$, where each $F_i$ is an induced subtournament of $F$ (possibly empty) and each $K(v_i)$ is the singleton graph with vertex $v_i$. Given this, for $\Phi$ to preserve adjacency and non--adjacency, it is also necessary that 
\begin{align}
 (f\circ \Phi )(V(F_0))<(f\circ \Phi )(v_1)<\dots < (f\circ \Phi )(v_p) < (f\circ \Phi )(V(F_p)), \label{eq:Qverts}
\end{align}
where if $X,Y$ are sets, we by $X<Y$ mean that $x<y$ for all $x\in X, y\in Y$. The probability of this event, conditional on the events in \eqref{eq:Ap} and \eqref{eq:Iverts}, and on the values $r_i:=(f\circ \Phi)(v_i)$, $i=1,\dots, p$, is
\begin{align}
& \frac{1}{\left(\sum_{i\in \cQ} v(G_{i,m}) \right)_{v(F)-p}} \prod_{i=0}^{p}\left(\sum_{j\in (r_i,r_{i+1})\cap \cQ}v(G_{j,m}) \right)_{v(F_i)} \\
&\qquad \sim \frac{\left(v(G_m) \right)^{v(F)-p}}{\left(\sum_{i\in \cQ} v(G_{i,m}) \right)^{v(F)-p}}\prod_{i=0}^{p}\left(\sum_{j\in (r_i,r_{i+1})\cap \cQ}\frac{v(G_{j,m})}{v(G_m)} \right)^{v(F_i)} \\
& \qquad \qquad \to \frac{1}{\alpha ^{v(F)-p}}\prod_{i=0}^p \alpha_{(r_i,r_{i+1})}^{v(F_i)}, \label{eq:glimit4}
\end{align}
as $m\to \infty$. 

Finally, conditional on the events in \eqref{eq:Ap}, \eqref{eq:Iverts} and \eqref{eq:Qverts}, for $\Phi$ to preserve adjacency and non--adjacency, it is necessary and sufficient that it does so on each restriction $\Phi|_{V(F_i)} : V(F_i)\to V(G_{(r_{i-1},r_i),m})$, where $G_{(a,b),m}$ denotes the induced subgraph of $G_m$ with vertex set $\{v\in V(G_m) \ : \ a<f(b)<c, f(b)\in \cQ \}$. This occurs with (conditional) probability
\begin{align}
 \prod_{i=0}^p t_{ind}(F_i,G_{(r_i,r_{i+1}),m})\to \prod_{i=0}^pt_{ind}(F_i,\Gamma_{(r_i,r_{i+1})})
\end{align}
as $m\to \infty$. 

It remains to integrate over the possible values of $r_i=(f\circ \Phi)(v_i)$, $i=1,\dots, p$, sum out over all decompositions $\cP(F,p)$ and sum out over all $p=0,\dots, v(F)$, and multiply together the above probabilities. Isolating the term for $p=0$ we obtain
\begin{align}\label{eq:glimit5}
  \E[t_{\text{ind}}(F,G_m)]&\to \alpha^{v(F)}t_{\text{ind}}\left(F,\bdirsum{\frac{\alpha_i}{\alpha} \Gamma_i}{i\in \cQ} \right) \\
&\quad +\sum_{p=1}^{v(F)}\sum_{\cP(F,p)}\binom{v(F)}{p}(1-\alpha)^p\int_{\Delta\cI^p} \prod_{i=0}^{p} \alpha_{(r_i,r_{i+1})}^{v(F_{i})} t_{\text{ind}}\left(F_{i}, \Gamma_{(r_i,r_{i+1})}\right) \prod_{i=1}^p \mu_{\cI}(dr_i)
\end{align}

Denote by $\dag$ the right hand side of \eqref{eq:glimit5}. Comparing this to \eqref{eq:vee1}, we see that if we take $W_i$ to be any tournament kernel representing $\Gamma_i$ for $i\in \cQ$, that the direct sum $W=\left(\bdirsum{\alpha_iW_i}{i\in \cQ},\cI,\eta \right)$ satisfies $t_{ind}(F,W)=\dag$. Now let $\Gamma$ be the unique (deterministic) tournament limit representing $W$, so that $t_{ind}(F,\Gamma)=\dag$. By Corollary 3.2 of \cite{DiaconisJanson}, it follows from $\E[t_{ind}(F,G_m)]\to t(F,\Gamma)$, for any tournament $F$, that $G_m\xrightarrow{p} \Gamma$. Now define
\begin{align}
 \left(\bdirsum{\alpha_i\Gamma_i}{i\in \cQ},\cI,\eta \right):= \Gamma,
\end{align}
and note that this is well--defined by uniqueness of limits and the fact that $\Gamma$ is non--random.

\end{proof}

We mention a few special cases of the formulae in Theorem \ref{thm:glimit1} and Theorem \ref{thm:glimit2}.

\begin{itemize}
 \item If $F$ is irreducible, then it has only the trivial decomposition $F=F$, so the formula in Theorem \ref{thm:glimit1} reduces to
\begin{align}
t_{ind}(F,\Gamma)=\sum_{i\in \cQ}\alpha_i^{v(F)}t_{ind}(F,\Gamma_i). 
\end{align}

\item If $\alpha=\sum_{i\in \cQ}\alpha_i=1$, then the only term that remains in \eqref{eq:glimit2} is the first, so we have
\begin{align}
t_{\text{ind}}(F,\Gamma)=\alpha^{v(F)}t_{\text{ind}}\left(F,\bdirsum{\frac{\alpha_i}{\alpha} \Gamma_i}{i\in \cQ} \right) = t_{\text{ind}}\left(F,\bdirsum{\alpha_i \Gamma_i}{i\in \cQ} \right),
\end{align}
to which Theorem \ref{thm:glimit1} is applicable.

\item If the decomposition of $F$ in Theorem \ref{thm:tourdecomp1} has no singleton components, then the set $\cP(F,p)$ consists only of the trivial decomposition $F=F$, so again the only term that remains in \eqref{eq:glimit2} is the first and so
\begin{align}
t_{\text{ind}}(F,\Gamma)=\alpha^{v(F)}t_{\text{ind}}\left(F,\bdirsum{\frac{\alpha_i}{\alpha} \Gamma_i}{i\in \cQ} \right).
\end{align}

\item If $\alpha=\sum_{i\in \cQ}\alpha_i=0$, then always $\alpha_{(r_i,r_{i+1})}=0$, so, by convention,
\begin{align}
 \alpha_{(r_i,r_{i+1})}^{v(F_i)} = \begin{cases} 0 & \text{ if } v(F_i)>0 \\ 1 & \text{ if } v(F_i)=0. \end{cases}
\end{align}
Hence only the term with $p=V(F)$ can give a non--zero contribution to the sum (else there is some $F_i$ with $v(F_i)>0$, introducing a zero factor to each summand). Therefore
\begin{align}
t_{\text{ind}}(F,\Gamma)=\sum_{\cP(F,v(F))}\int_{\Delta\cI^{v(F)}}\prod_{i=1}^{v(F)}\mu_{\cI}(dr_i)= \begin{cases} \frac{1}{k!}, & \text{ if } F=\sT_k \\ 0, & \text{ otherwise,} \end{cases}
\end{align}
where the final equality follows since $F$ can only have a decomposition into $v(F)$ singletons if $F$ is transitive. In particular, note that $\sum_{i\in \cQ}\alpha_i=0$ implies that $\Gamma$ is transitive (see Theorems \ref{thm:acyclic} and \ref{thm:transeq}).
\end{itemize}

This completes our definition of direct sums of tournament limits. The following result shows that the notions of direct sums agree for tournament limits and tournament kernels.

\begin{theorem}\label{thm:existence}
For each $i\in \cQ$, let $W_i$ be some tournament kernel representing $\Gamma_i$. Then $\left(\bdirsum{\alpha_iW_i}{i\in \cQ},\cI,\eta \right)$ represents $\left(\bdirsum{\alpha_i\Gamma_i}{i\in \cQ},\cI,\eta \right)$.
\end{theorem}

This follows immediately from the correspondence between the formulae in Theorems \ref{thm:densvee1} and \ref{thm:densvee2} and the formulae in Thereoms \ref{thm:glimit1} and \ref{thm:glimit2}.

%% file: master_irreducible.tex
%%%%%%%%%%%%%%%%%%%%%%%%%%%%%%%%%%%%%%%%%%%
\section{Irreducible and transitive limits}
%%%%%%%%%%%%%%%%%%%%%%%%%%%%%%%%%%%%%%%%%%%%
In this section, we seek to extend Theorem \ref{thm:irreducibletour} to the setting of tournament kernels. This will allow us to show how the irreducibility (to be made precise) of a tournament limit corresponds to the irreducibility of a representing tournament kernel and the irreducibility the induced random infinite tournament. 

\begin{definition}
 For any $B\subseteq \cS$ with $0<\mu(B)<1$, define the outneighbourhood of $B$ by 
\begin{align}
 N(B)=\{y\in \cS \ : \ \int_{B}W(x,y)d\mu(x)>0\}.
\end{align}
and define $N^m(B)=N(N^{m-1}(B))$, where $N^{1}(B)=N(B)$. For a singleton $x\in \cS$, define its outneighbourhood as 
\begin{align}
N(\{x\})=\{y \in \cS \ : \ W(x,y)>0 \}.
\end{align}
If $\mu(N(\{x\})$, we define $N^m(\{x\})=N^{m-1}(N(\{x\}))$ for $m>1$.
\end{definition}

The defining relation $W(x,y)+W(y,x)=1$ implies that $\mu(N(\{x\}))>0$ for almost all $x\in \cS$. We say that the kernel $W$ is \emph{strongly connected} if, for almost all $x\in \cS$, the set $A(x)=\bigcup_{i=1}^{\infty}N^i(\{x\})$ satisfies $\mu(A(x))=1$. Note also that $A\subseteq B$ implies $N(A)\subseteq N(B)$

In light of Theorem \ref{thm:irreducibletour}, we prove the following equivalence theorem for irreducible tournament kernels.

\begin{theorem} \label{thm:irreducibleW}
Let $W$ be a tournament kernel. The following statements are equivalent.
\begin{enumerate}
  \item $W$ is irreducible. \label{thm:irreducibleW1}
  \item $W$ is strongly connected. \label{thm:irreducibleW2}
  \item There does not exist a measurable subset $B\subseteq \cS$ with $0<\mu(B)< 1$ such that $\mu(N(B)\setminus B)=0$.\label{thm:irreducibleW3}
  \item There does not exist a measurable subset $B\subseteq \cS$ with $0<\mu(B)<1$ such that $\int_{B\times \cS}W(x,y)d\mu(x)d\mu(y)=\frac{\mu(B)^2}{2}$.\label{thm:irreducibleW4}
\end{enumerate}
\end{theorem}

\begin{proof}
\begin{description}
 \item[$(\ref{thm:irreducibleW1}) \Longrightarrow (\ref{thm:irreducibleW3})$]  
Suppose $W$ is irreducible and, for contradiction, that there exists some $B\subseteq \cS$ with $0<\mu(B)<1$ such that $\mu(N(B)\setminus B)=0$. Then, for $\mu$--all $y\notin N(B)$ we have $\int_{B}W(x,y)d\mu(x)=0$, whence $\int_{B\times (\cS\setminus N(B))}W(x,y)d\mu(x)d\mu(y)=0$. But then $W(x,y)=0$ for almost all $(x,y)\in B\times (\cS\setminus N(B))$. But $B\times (\cS\setminus B) \subseteq B\times (\cS\setminus N(B))$ (up to zero measure), so $W$ is reducible; a contradiction.

 \item[$(\ref{thm:irreducibleW1}) \Longleftrightarrow (\ref{thm:irreducibleW4})$] 
We prove the contrapositives. Suppose that there is such a set $B$. Then
\begin{align}
 \int_{B^2}W(x,y)d\mu(x)d\mu(y) =\frac{\mu(B)^2}{2}=\int_{B\times \cS}W(x,y)d\mu(x)d\mu(y),
\end{align}
which implies that $W(x,y)=0$ for almost all $(x,y)\in B\times (\cS\setminus B)$. Thus $W$ is reducible. The other direction is identical.
                                                                                                                                           
\item[$(\ref{thm:irreducibleW3}) \Longrightarrow (\ref{thm:irreducibleW2})$]
We prove the contrapositive statement. First, for almost every $x\in \cS$ we have $\mu(A(x))\geq \mu(N(\{x\}))>0$. Suppose $W$ is not strongly connected. Then there exists some non--null set of $x\in \cS$ for which $0<\mu(A(x))<1$. It follows from the definition of $A(x)$ that $\mu(N(A(x))\setminus A(x)=0$ for each such $x$. Take $B=A(x)$.

\item[$(\ref{thm:irreducibleW2}) \Longrightarrow (\ref{thm:irreducibleW1})$]
We prove the contrapositive statement, so suppose $W$ is reducible. Since $W$ is reducible, there exists $B\subseteq \cS$ with $0<\mu(B)<1$ such that $W(x,y)=0$ for a.e. $(x,y)\in B\times \cS\setminus B$. We prove by induction that there exists a non--null set of $x\in \cS$ and  such that $\mu(N^i(x)\setminus B)=0$ for all $i\geq 1$. For almost all $x\in B$ we have $\mu(N(x))>0$ and $\mu(N(x)\setminus B)=0$, so the base case holds. For the induction step, suppose $\mu(N^{i}(x)\setminus B)=0$. Then

\begin{align}
 \mu(N^{i+1}(x)\setminus B)&=\mu(N(N^i(x))\setminus B) \\
&=\mu(\{y\in \cS \ : \ \int_{N^i(x)}W(x,y)d\mu(x)>0 \}\setminus B) \\
&\leq \mu(\{y\in \cS \ : \ \int_{B}W(x,y)d\mu(x)>0 \}\setminus B) \\
&=\mu(N(B)\setminus B)=0,
\end{align}
which completes the induction step. But this implies that $\mu\left(\bigcup_{i=1}^{\infty}N^{i}(x)\setminus B \right)=0$, so $0<\mu\left(\bigcup_{i=1}^{\infty}N^{i}(x)\right)\leq \mu(B)<1$.
 \end{description}
\end{proof}

We now define what we mean by irreducible and transitive tournament limits, and show that these are well--defined properties, in the sense that a tournament limit is irreducible (transitive) if and only if its representing kernel is irreducible (transitive) if and only if its induced infinite random graph is irreducible (transitive).

 Given $0<r<1$ and two tournament limits $\Gamma,\Gamma'$ , we define $r \Gamma \dirsum{}{} (1-r)\Gamma'$ as the direct sum $\bdirsum{\alpha_i\Gamma_i}{i\in \{0,r\}}$ where $(\alpha_0,\alpha_r)=(r,1-r),(\Gamma_0,\Gamma_r)=(\Gamma,\Gamma')$ (and $\eta=\id : \{0,r\}\to [0,1]$).

\begin{definition}
 A tournament limit $\Gamma$ is said to be \emph{reducible} if there exists $r>0$ and tournament limits $\Gamma_1,\Gamma_2$ such that $\Gamma=r\Gamma_1 \dirsum{}{} (1-r)\Gamma_2$.
\end{definition}

The proof of the following theorem is similar to a few results from \cite{Janson2008}, so we omit the proof and refer the reader to the relevant sections of that paper. 
\begin{theorem} \label{thm:eqirr}
 Let $\Gamma$ be a tournament limit represented by a tournament kernel $W$. Then the following are equivalent.
\begin{enumerate}[(a)]
 \item $\Gamma$ is irreducible. \label{thm:eqirr1}
  \item $W$ is irreducible. \label{thm:eqirr2}
  \item $G(\infty,\Gamma)=G(\infty,W)$ is a.s. irreducible. \label{thm:eqirr3}
\end{enumerate}
\end{theorem}

\begin{proof}
 \begin{description}
  \item[$\ref{thm:eqirr1}\Longleftrightarrow \ref{thm:eqirr3}$] This corresponds to Theorem 1.19 of \cite{Janson2008}, which uses Lemma 5.2 of the same paper. The only difference is that we need directed paths in both directions between any pair of vertices, but the same proof works in our situation.
  \item[$\ref{thm:eqirr1}\Longleftrightarrow \ref{thm:eqirr2}$] This corresponds to Theorem 1.16 of \cite{Janson2008}.  
 \end{description}
\end{proof}

If $\Gamma$ is reducible, it does not follow that $G(n,\Gamma)$ is reducible. However, one can show that it is reducible with probability at least $1-e^{-\Omega(n)}$. Similarly, if $\Gamma$ is irreducible, it does not follow that $G(n,\Gamma)$ is irreducible. For instance, let $\Gamma$ be the tournament limit represented by the kernel in Figure \ref{fig:irreduciblekernel} with the subkernels on the diagonal being the transitive kernels. In this case $G(n,\Gamma)$ is reducible with probability at least $3(2/3)^n>0$.

% \begin{proposition}
%  If $\Gamma$ is reducible, then $G(n,\Gamma)$ is reducible with probability at least $1-e^{O(n)}$.\footnote{Can we say something about the probability that $G(n,\Gamma)$ is irreducible, given that $\Gamma$ is irreducible? Is it also exponentially close to $1$? For instance, consider the kernel $W$ in figure \ref{fig:irreduciblekernel} with the subkernels on the diagonal being the transitive kernels. Then $G(n,W)$ is reducible with probability $3(2/3)^n>0$, which tends to zero exponentially quickly.}
% \end{proposition}

% \begin{proof}
%  If $\Gamma$ is reducible, then there exists exists $\alpha\in (0,1)$ and tournament limits $\Gamma_1,\Gamma_2$ such that $\Gamma=\alpha \Gamma_1\oplus (1-\alpha)\Gamma_2$. The random graph $G(n,\Gamma)$ is reducible with probability at least $1-(\alpha^n+(1-\alpha)^n)=1-e^{O(n)}$. 
% \end{proof}

Having shown that tournament kernels and tournament limits agree on the notion of irreducibility, we show that the same holds for transitivity. In this case however, the equivalence extends to the finite random graphs $G(n,\Gamma)$. 

\begin{definition}
 A tournament limit $\Gamma$ is said to be \emph{transitive} if $\Gamma=r\Gamma\dirsum{}{} (1-r)\Gamma$ for all $r\in [0,1]$.
\end{definition}

\begin{theorem}\label{thm:transeq}
 Let $\Gamma$ be a tournament limit represented by a tournament kernel $W$. Then the following are equivalent.
\begin{enumerate}[(i)]
 \item $\Gamma$ is transitive. \label{thm:transeq1}
 \item $W$ is transitive. \label{thm:transeq2}
 \item $G(\infty,\Gamma)=G(\infty,W)$ is a.s. transitive. \label{thm:transeq3}
 \item For each $n\geq 1$, $G(n,\Gamma)=G(n,W)$ is a.s. transitive. \label{thm:transeq4}
\end{enumerate}
\end{theorem}
\begin{proof}
\ref{thm:transeq1} $\Longrightarrow$ \ref{thm:transeq2}.
 By Theorem \ref{thm:glimit1} (using the fact that there is no non--trivial decomposition of $\sC_3$) we have $t_{ind}(\sC_3,\Gamma)=r^3 t_{ind}(\sC_3,\Gamma)+(1-r)^3t_{ind}(\sC_3,\Gamma)$ for any $r\in [0,1]$. This can only hold if $t_{ind}(\sC_3,\Gamma)=0$, which implies that $t_{ind}(\sC_3,W)=0$. Theorem \ref{thm:acyclic} now says that $W$ is transitive. 

\ref{thm:transeq2} $\Longrightarrow$ \ref{thm:transeq1}.
 Take any $r\in [0,1]$. By Theorem \ref{thm:acyclic} we may assume $W=\mathbbm{1}_{\{x\leq y\}}:([0,1]^2,\Leb)\to [0,1]$. For this choice it holds that $W=r W \dirsum{}{} (1-r)W$. Hence $\Gamma$ is also represented by $r W \dirsum{}{} (1-r)W$, but so is $r \Gamma \dirsum{}{}(1-r)\Gamma$ by Theorem \ref{thm:existence}. Uniqueness of tournament limits gives that $\Gamma=r \Gamma \dirsum{}{}(1-r)\Gamma$.

\ref{thm:transeq3} $\Longrightarrow$ \ref{thm:transeq4}.
Trivial.

\ref{thm:transeq4} $\Longrightarrow$ \ref{thm:transeq2}. Recall that $G(n,\Gamma)\to \Gamma$. Since $G(n,\Gamma)$ is a.s. transitive, Lemma \ref{lem:acycliconv} implies that $\Gamma$ is represented by $W_T:([0,1]^2,\Leb)\to [0,1]$ given by $W_T=\mathbbm{1}_{\{x\geq y\}}$. Therefore $W$ and $W_T$ are equivalent, so $W$ is transitive by Theorem \ref{thm:acyclic}.

\ref{thm:transeq2} $\Longrightarrow$ \ref{thm:transeq3}. 
Use the fact that $W$ is equivalent to $W_T$.
\end{proof}

Since any transitive tournament limit is represented by the kernel $\mathbbm{1}_{\{x\geq y\}}$, the following corollary follows.

\begin{corollary}
There is a unique transitive tournament limit.
\end{corollary}

%% file: master_degdist.tex
\section{Decompositions of tournament limits}

Recall that each tournament kernel $W:\cS^2 \to [0,1]$ has a unique decomposition into irreducible components with a ``transitive`` set interlaced between these components. Uniqueness of the decomposition of tournament limits does not (easily) follow from the uniqueness of the decomposition of tournament kernels. This is because each tournament limit can be represented by many different tournament kernels, and it is not clear that two equivalent tournament kernels have ``the same'' decomposition. We use Theorem \ref{thm:tourdecomp1} along with the uniqueness of the distribution of $G(\infty,\Gamma)$ to get around this problem.

\begin{theorem}\label{thm:limitdecomp}
Each tournament limit $\Gamma \in \widehat{\cT}$ can be decomposed as a direct sum $\left(\bdirsum{\alpha_i\Gamma_i}{i\in \cQ},\cI, \mu_{\cI}, \eta \right)$ where $\cQ$ is countable, $(\cI,\mu_{\cI})$ is a non--atomic probability space, each $\Gamma_i$, $i\in \cQ$ is irreducible, and $\eta:\cQ\cup \cI\to [0,1]$ is such that $\{[\eta(i),\eta(i)+\alpha_i)\}_{i\in \cQ}\cup \eta(\cI)$ is a partition of $[0,1]$ (up to null sets) and $\eta|_{\cI}$ is injective (up to null sets).

This decomposition is unique in the following sense. If $\left(\bdirsum{\alpha'_i\Gamma'_i}{i\in \cQ'},\cI',\mu'_{\cI'},\eta' \right)$ is another decomposition of $\Gamma$, then there exists a bijection $f:\cQ \to \cQ'$ such that $\eta(i)<\eta(j)$ if and only if $\eta'(f(i))<\eta'(f(j))$ and such that $\alpha_i=\alpha'_{f(i)}$, $\Gamma_i=\Gamma'_{f(i)}$ and 
\begin{align}
 \mu_{\cI}\left(\{j\in \cI \ : \ \eta(j)<\eta(i) \} \right) = \mu'_{\cI'}\left(\{j\in \cI' \ : \ \eta(j)<\eta(f(i)) \} \right)
\end{align}
for any $i\in \cQ$.
\end{theorem}

Note that the last equality uniquely determines how $\eta(\cI)$ is interlaced with $\eta(\cQ)$.

\begin{proof}
 \textit{Existence of decomposition.}
This follows by taking the decomposition in Theorem \ref{thm:Wdecomp} (with the map $\eta$ coming from there) and applying Theorem \ref{thm:existence}.

% \begin{proof}
% \textbf{Existence.} 
%  Let $\Gamma$ be represented by the kernel $W :(\cS^2,\mu)\to [0,1]$. This has a decomposition like the one given by Theorem \ref{thm:Wdecomp}. Using the same notation, we have that $\Gamma$ is also represented by $\left(\bigoplus_{i\in \cQ}\alpha_iW_i,\cI\right)$. Let $\Gamma_i$ be the tournament limit represented by $W_i$. By combining Theorems \ref{thm:glimit1} and \ref{thm:glimit2} with Theorems \ref{thm:densvee1} and \ref{thm:densvee2} we have that the tournament limit $\left(\bdirsum{\alpha_i \Gamma_i}{i\in \cQ},\cI\right)$ is represented by $\left(\bdirsum{\alpha_i W_i}{i\in \cQ},\cI\right)$. Hence $\Gamma=\left(\bdirsum{\alpha_i \Gamma_i}{i\in \cQ},\cI\right)$, since these tournament limits are represented by the same kernel. But $\Gamma_i$ is irreducible if and only if $W_i$ is irreducible according to Theorem \ref{thm:eqirr}. By Theorem \ref{thm:Wdecomp}, this implies that $\Gamma_i$ is irreducible for $i\in \cQ$. This proves existence of the decomposition. 
% \end{proof}

 \textit{Uniqueness.}
 We first prove that $\Gamma$ uniquely determines the pairs $(\alpha_i,\Gamma_i)_{i\in \cQ}$ up to order isomorphism of $\eta(\cQ)$. 

By the existence part of this proof, $\Gamma$ is represented by the tournament kernel $W=\left(\bigoplus_{i\in \cQ}\alpha_i,W_i,\cI,\mu_{\cI} \right)$, where $W_i$ represents $\Gamma_i$ for all $i\in \cQ$. Moreover, the $W_i$ are irreducible by Theorem \ref{thm:eqirr}. 

Recall that $\Gamma$ determines the distribution of $G(\infty,\Gamma)=G(\infty,W)$, and recall the construction of this random graph in Section 2. Let $V_i=\{k \ : \ X_k\in \cS_i \}$, $i\in \cQ$, and let $V_{\cI}=\{k \ : \ X_k\in \cI\}$. By the strong law of large numbers, 
\begin{align}
 &\lim_{n\to \infty} \frac{|V_i\cap [n]|}{n}=\pP[X_1\in \cS_i]=\alpha_i, \qquad \qquad i\in \cQ \\
 &\lim_{n\to \infty} \frac{|V_{\cI}\cap [n]|}{n}=\pP[X_1\in \cI]=1-\sum_{i\in \cQ}\alpha_i.
\end{align}
In particular, this implies that $|V_i|=\infty$ a.s. for any $i\in \cQ$, and provided $\sum_{i\in \cQ}\alpha_i<1$, that $|V_{\cI}|=\infty$ a.s. For any $i\in \cQ$, let $G_i$ be the induced subgraph $G(\infty,W)|_{V_i}$. Its vertices are chosen independently from $\cS_i$ according to $\mu_i$, whence $G(\infty,W)|_{V_i} \stackrel{d}{=} G(\infty,W_i) = G(\infty,\Gamma_i)$ (after relabeling of the vertices). Each $G_i$, $i\in \cQ$, is therefore irreducible by Theorem \ref{thm:eqirr}, so we obtain a decomposition 
\begin{align}
 G(\infty,\Gamma)=\bigoplus_{i\in \eta(\cQ) \cup \eta(V)}G_i
\end{align}
where $V=\{X_k \ : \ k\in V_{\cI} \}$, such that each $G_i$, $i\in \cQ$, is irreducible and has infinitely many vertices and each $G_i$, $i\in V$, consists of a single vertex. By Theorem \ref{thm:tourdecomp1}, this decomposition is unique up to order isomorphism of $\eta(\cQ \cup V)$. In particular, forgetting about $V$ (this presents no problem, since $V$ and $\cQ$ are different in the sense that $\cQ$ enumerates infinite components, while $V$ enumerates finite components), the set $\eta(\cQ)$ is unique up to order isomorphism. Hence $(\alpha_i, \Gamma_i)_{i\in \cQ}$ is uniquely determined by $\Gamma$, up to order isomorphism of the set $\eta(\cQ)$. 

Suppose now that we have two decompositions 
\begin{align}
 \left(\bigoplus_{i\in \cQ}\alpha_i\Gamma_i,\cI, \mu_{\cI},\eta \right) 
\qquad \text{and}
\qquad 
\left(\bigoplus_{i\in \cQ'}\alpha'_i\Gamma'_i,\cI',\mu'_{\cI'},\eta'\right)
\end{align}
both equal to some tournament limit $\Gamma$. By the above there exists a bijection $f:\cQ \to \cQ'$ such that $\alpha'_{f(i)}=\alpha_i$ and $\Gamma'_{f(i)}=\Gamma_i$ for all $i\in \cQ$. Moreover, $f$ preserves order in the sense that $\eta(i)<\eta(j)$ if and only if $\eta'(f(i))<\eta'(f(j))$. 
These decompositions induce decompositions of $G(\infty,\Gamma)$ like above. For any $i\in \cQ$, denote by $A_i$ the event that $(v_1,v)\in E(G(\infty,\Gamma))$ for all $v\in V(G_i)$, and for any $j\in \cQ'$ let $A'_j$ denote the event that $(v_1,v)\in E(G(\infty,\Gamma))$ for all $v\in V(G'_j)$. In both these cases $v_1$ denotes the vertex labelled $1$ in $G(\infty,\Gamma)$. By the construction of $G(\infty,\Gamma)$, it follows that
\begin{align}
 \pP[A_i]&=\sum_{\substack{j\in \cQ \\ \eta(j)<\eta(i)}}\alpha_j + (1-\alpha)\mu_{\cI}\{j\in \cI \ : \ \eta(j)<\eta(i) \} \\
 \pP[A'_{f(i)}]&=\sum_{\substack{j\in \cQ' \\ \eta'(j)<\eta'(f(i))}}\alpha'_j + (1-\alpha')\mu'_{\cI'}\{j\in \cI' \ : \ \eta'(j)<\eta'(f(i)) \},
\end{align}
where $\alpha=\sum_{i\in \cQ}\alpha_i$ and $\alpha'=\sum_{i\in \cQ'}\alpha_i'$. We have shown already that $\alpha=\alpha'$. Since $f$ is a bijection that preserves order of the pair $(\eta,\eta')$, it holds that 
\begin{align}
 \sum_{\substack{j\in \cQ' \\ \eta'(j)<\eta'(f(i))}}\alpha'_j = \sum_{\substack{j\in \cQ \\ \eta(j)<\eta(i)}}\alpha_i,
\end{align}
for any $i\in \cQ$. Therefore
\begin{align}
 \mu_{\cI}(j\in \cI \ : \ \eta(j)<\eta(i) \} = \mu'_{\cI'}(j\in \cI' \ : \ \eta'(j)<\eta'(f(i)) \},
\end{align}
for any $i\in \cQ$.
\end{proof}

%% file: master_acknowledgements.tex
\section*{Acknowledgements}
The author wishes to thank Svante Janson for meticulously reading through and suggesting significant improvements to numerous previous versions of this manuscript.

%% file: master.bbl
\begin{thebibliography}{10}

\bibitem{Boeckner}
D.~Boeckner.
\newblock {\em Directed Graph Limits and Directed Threshold Graphs}.
\newblock PhD thesis, University of Nebraska, 2013.

\bibitem{Borgs2008}
C.~Borgs, J.~T. Chayes, L.~Lov{\'a}sz, V.~T. S{\'o}s, and K.~Vesztergombi.
\newblock Convergent sequences of dense graphs {I}: Subgraph frequencies,
  metric properties and testing.
\newblock {\em Advances in Mathematics}, 219(6):1801--1851, 2008.

\bibitem{Borgs2012}
C.~Borgs, J.~T. Chayes, L.~Lov{\'a}sz, V.~T. S{\'o}s, and K.~Vesztergombi.
\newblock Convergent sequences of dense graphs {II}: Multiway cuts and
  statistical physics.
\newblock {\em Annals of Matematics}, 176(1):151--219, 2012.

\bibitem{CourcelleDelhomme2008}
B.~Courcelle and C.~Delhomm{\'e}.
\newblock The modular decomposition of countable graphs. {D}efinition and
  construction in monadic second order logic.
\newblock {\em Theoretical Computer Science}, 394:1--38, 2008.

\bibitem{DiaconisHolmesJanson}
P.~Diaconis, S.~Holmes, and S.~Janson.
\newblock Threshold graph limits and random threshold graphs.
\newblock {\em Internet Math.}, 5(3):267--320, 2008.

\bibitem{DiaconisJanson}
P.~Diaconis and S.~Janson.
\newblock Graph limits and exchangeable random graphs.
\newblock {\em Rend. Mat. Appl. (7)}, 7(1):33--61, 2008.

\bibitem{HatamiNorine2013}
H.~Hatami and S.~Norine.
\newblock The entropy of random-free graphons and properties.
\newblock {\em Combinatorics, Probability and Computing}, 22(04):517--526,
  2013.

\bibitem{HellNesetril2004}
P.~Hell and J.~Ne\v{s}et\v{r}il.
\newblock {\em Graphs and Homomorphisms}.
\newblock Oxford Lecture Series in Mathematics and Its Applications. Oxford
  University Press, 2004.

\bibitem{Janson2008}
S.~Janson.
\newblock Connectedness in graph limits.
\newblock Technical Report~8, Uppsala University, 2008.
\newblock ar{X}iv:0802.3795.

\bibitem{Janson2011b}
S.~Janson.
\newblock Poset limits and exchangeable random posets.
\newblock {\em Combinatorica}, 31(5):529--563, 2011.

\bibitem{Janson2011bpreprint}
S.~Janson.
\newblock Poset limits and exchangeable random posets.
\newblock {\tt arXiv:0902.0306}, 2011.
\newblock Preprint.

\bibitem{Janson2011}
S.~Janson.
\newblock Graph limits and hereditary properties.
\newblock {\em European J. Combinatorics}, 2015.
\newblock http://dx.doi.org/10.1016/j.ejc.2015.07.010.

\bibitem{LiebLoss2001}
E.~H. Lieb and M.~Loss.
\newblock {\em Analysis}.
\newblock American Mathematical Society, second edition, 2001.

\bibitem{Lovaszbook}
L.~Lov{\'a}sz.
\newblock {\em Large Networks and Graph Limits}.
\newblock Colloquium Publications. American Mathematical Society, 2012.

\bibitem{LovaszSzegedy2006}
L.~Lov{\'a}sz and B.~Szegedy.
\newblock Limits of dense graph sequences.
\newblock {\em J. Comb. Theory B}, 96(6):933--957, 2006.

\bibitem{LovaszSzegedy2010}
L.~Lov{\'a}sz and B.~Szegedy.
\newblock Regularity partitions and the topology of graphons.
\newblock In {\em An irregular mind: Szemer{\'e}di is 70}. Springer Verlag,
  2010.

\bibitem{LovaszSzegedy2011}
L.~Lov{\'a}sz and B.~Szegedy.
\newblock Finitely forcible graphons.
\newblock {\em J. Comb. Theory}, B 101(5):269--301, 2011.

\bibitem{Moon1968}
J.~W. Moon.
\newblock {\em Topics on Tournaments}.
\newblock Holt, Rinehart and Winston, Inc., 1968.

\end{thebibliography}
